\documentclass[12pt]{article}

\usepackage{amsmath} 
\usepackage{amsthm}
\usepackage{amssymb}
\usepackage{hyperref} 
\hypersetup{pdfborder = {0 0 0}}
\usepackage{color}
\usepackage{caption}
\usepackage{subcaption}
\usepackage{mathtools}

\usepackage{tikz}
\usepackage{tkz-graph}
\usepackage{enumitem,comment}
\usepackage{xparse}

\usetikzlibrary{math}
\usetikzlibrary{calc}
\usetikzlibrary{positioning}

\counterwithin{equation}{section}
\newtheorem{theorem}[equation]{Theorem}
\newtheorem{lemma}[equation]{Lemma}
\newtheorem*{corollary}{Corollary} 
\newtheorem{conjecture}[equation]{Conjecture}
\numberwithin{equation}{section}

\theoremstyle{definition}
\newtheorem*{definition}{Definition} 
\newtheorem*{example}{Example}

\newtheoremstyle{examplen}
  {\topsep}
  {\topsep}
  {}
  {}
  {}
  {}
  {.5em}
  {\hspace{-6.5pt}}
\theoremstyle{examplen}
\newtheorem*{examplen}{Example}

\theoremstyle{remark}
\newtheorem*{remark}{Remark}

\newcommand{\FT}{\mathrm{FT}}
\DeclareMathOperator\forest{\mathcal{F}}
\newcommand{\ttrip}{\mathrm{\mathcal{T}}}
\DeclareMathOperator\type{type}
\newcommand{\sign}{\mathrm{sign}}
\DeclareMathOperator\E{E}
\DeclareMathOperator\tbreak{break}

\DeclareMathOperator\shiftbreak{shiftbreak}

\DeclareMathOperator\tjoin{join}

\DeclareMathOperator\shiftjoin{shiftjoin}

\DeclareMathOperator\secondjoin{secondjoin}
\DeclareMathOperator\secondbreak{secondbreak}
\DeclareMathOperator\len{len}
\DeclareMathOperator\sort{sort}

\DeclareMathOperator\last{last}
\DeclareMathOperator\first{first}

\DeclareMathOperator\comp{comp}
\DeclareMathOperator\rotatejoin{rotatejoin}
\DeclareMathOperator\rotatebreak{rotatebreak}

\newcommand{\scolon}{;\;}

\tikzset{every picture/.append style={font=\footnotesize}}

\definecolor{red}{rgb}{0.75,0.0,0.0}
\definecolor{blue}{rgb}{0.0,0.0,1.0}
\definecolor{black}{rgb}{0.0,0.0,0.0}
\definecolor{lightblue}{rgb}{0.0,0.5,1.0}
\definecolor{green}{rgb}{0.0,1.0,0.0}
\definecolor{darkgreen}{rgb}{0.0,0.5,0.0}
\definecolor{purple}{rgb}{0.5,0.0,1.0}
\NewDocumentCommand{\vertex}{O{black} O{above} m O{(0,0)} m O{}}{\path let \p1 = #5, \p2=#4 in node[style={draw,fill,color=#1,circle,scale=0.75},label={[text=#1]#2:$#3$}] (v#6#3) at (\x1+\x2,\y1+\y2){};}
\tikzset{fontscale/.style = {font=\relsize{#1}}
    }

\makeatletter
\newcommand{\oset}[3][0ex]{%
  \mathrel{\mathop{#3}\limits^{
    \vbox to#1{\kern-2\ex@
    \hbox{$\scriptstyle#2$}\vss}}}}
\makeatother
\usepackage{authblk}

\title{Adjacent cycle-chains are $e$-positive}

\author[1]{Foster Tom\thanks{ftom@mit.edu}}
\author[2]{Aarush Vailaya\thanks{25aarushv@students.harker.org}}

\affil[1]{Department of Mathematics, MIT, Cambridge MA}
\affil[2]{Harker School, Santa Clara CA}

\let\svlabel\label
\let\svref\ref
\renewcommand\label[2][\relax]{%
  \svlabel{#2}%
  \ifx\relax#1\else
      \expandafter\gdef\csname custom#2\endcsname{#1}%
  \fi
}
\renewcommand\ref[2][\relax]{%
  \ifx s#1\csname custom#2\endcsname\else\svref{#2}\fi
}

\newcommand\labelAndRemember[2]
  {\expandafter\gdef\csname labeled:#1\endcsname{#2}%
   \label{#1}#2}
\newcommand\recallLabel[1]
   {\csname labeled:#1\endcsname\tag{\ref{#1}}}
\begin{document}

\date{}

\maketitle

\begin{abstract}
We describe a way to decompose the chromatic symmetric function as a positive sum of smaller pieces. We show that these pieces are $e$-positive for cycles. Then we prove that attaching a cycle to a graph preserves the $e$-positivity of these pieces. From this, we prove an $e$-positive formula for graphs of cycles connected at adjacent vertices. We extend these results to graphs formed by connecting a sequence of cycles and cliques.
\end{abstract}

\section{Introduction}

For a graph $G$, the chromatic symmetric function $X_G(\boldsymbol{x})$ was first defined by Stanley in \cite{stanley1995symmetric}, as a function on $\boldsymbol{x}=x_1,x_2,\ldots$, an infinite sequence of variables. The chromatic symmetric function can be written in many bases, and one particular field of interest is the positivity of the coefficients in these bases. One specific basis of interest is the elementary basis, or the $e$-basis. The famous Stanley--Stembridge conjecture in~\cite[Conjecture 5.5]{STANLEY1993261} claims that incomparability graphs of $\mathbf{(3+1)}$-free partially ordered sets are $e$-positive, meaning the coefficients of the chromatic symmetric functions of these graphs in the $e$-basis are all positive. This conjecture is closely related to the immanants of Jacobi-Trudi matrices \cite{STANLEY1993261}, the cohomology of Hessenberg varieties \cite{abreu2023splitting}, and the characters of Kazhdan--Lusztig elements of the Hecke algebra \cite{abreu2022parabolic}. Gasharov in \cite{GASHAROV1996193} proved that such graphs are Schur-positive, a weaker condition than $e$-positivity. Guay-Paquet in \cite{guaypaquet2013modular} reduced the Stanley--Stembridge conjecture to proving that all natural unit interval graphs are $e$-positive. The converse is not true, as there are many non-unit interval graphs (such as cycles) that are $e$-positive. It is generally unknown when a graph is $e$-positive or not, but many papers have proven the $e$-positivity of specific families of graphs and derived explicit formulas for certain families of graphs \cite{Alexandersson_2018,aliniaeifard2021chromatic,MR4417181,Cho2017OnEA,dahlberg2019triangularladderspd2epositive,tom2024signed,wang2024cyclechordsepositive}.

Many results related to $e$-positivity have been achieved working with certain generalizations or alternate versions of the chromatic symmetric function, such as the chromatic symmetric function in non-commuting variables \cite{Gebhard2001} or the quasisymmetric refinement \cite{shareshian2016chromaticquasisymmetricfunctions}.

The first author in \cite{tom2024signed} found a new method of finding the chromatic symmetric function for certain unit interval graphs, using objects called forest triples. With these he proved that $K\!$-chains, which are cliques connected at single vertices, are $e$-positive. We will use a similar method to prove that adjacent cycle chains are also $e$-positive. More precisely, our paper is structured as follows.

Section~\ref{sec:basic} introduces the definitions to describe the chromatic symmetric function, the $e$-basis, unit interval graphs, and the Stanley--Stembridge conjecture. In Section~\ref{sec:ftrip}, we present a way to calculate the $e$-expansion of the chromatic symmetric function of a graph using objects called forest triples. In Section~\ref{sec:sri} we find a way to decompose the chromatic symmetric function into multiple pieces, which we conjecture are all $e$-positive. Section~\ref{sec:cgi} presents a new proof of the already-known $e$-positivity of cycles, using forest triples and involutions. In Section~\ref{sec:cycchain}, we prove that given an involution on ``cycle+tree" graphs, attaching a cycle to a certain graphs at a single vertex preserves $e$-positivity. From this, we get that adjacent cycle chains are $e$-positive. Additionally, we derive explicit formulas for certain graphs, such as the graph of two cycles connected at a single vertex, $C_a+C_b$, with
\begin{multline}\labelAndRemember{eqno:C_a+C_b}
    {X_{C_a+C_b}(\boldsymbol{x})=\sum_{\substack{\alpha\models a,\beta\models b+\alpha_1\\\len(\beta)\geq 2,\beta_1\leq \alpha_1,\beta_2\geq \alpha_1}}(\alpha_1-1)\cdots (\alpha_l-1)\cdot(\beta_2-\beta_1+1)(\beta_1+\beta_2-\alpha_1-1)\cdot\\[-20pt](\beta_3-1)\cdots (\beta_l-1)\cdot e_{\sort(\alpha\setminus \alpha_1\cdot \beta\setminus\beta_1 \cdot (\beta_1-1))}}.
\end{multline}
Finally, Section~\ref{sec:cpti} proves the involution on ``cycle+tree" graphs used in Section~\ref{sec:cycchain}, through  similar methods as Section~\ref{sec:cgi}.
\section{Background}\label{sec:basic}
In this paper, $G$ always references a loopless non-directed graph with $n$ vertices labeled $1$ through $n$, and with a fixed total ordering on the edges. The vertex set of a graph is denoted as $V(G)$ and edge set as $E(G)$. The number of vertices in a graph is denoted as $|G|$. The clique $K_n$ is the graph with $n$ vertices and an edge between every pair of vertices.
\begin{definition}
A \emph{proper coloring} of graph $G$ is a function $\kappa:V(G)\rightarrow \mathbb{N}$ such that if $(i,j)\in E(G)$, then $\kappa(i)\neq \kappa(j)$.
\end{definition}
\begin{definition}
    Let $\boldsymbol{x}=(x_i\;\text{for all} \;i\in\mathbb{N})$ be an infinite tuple of variables. Then, the \emph{chromatic symmetric function} is defined as \begin{equation*}X_G(\boldsymbol{x})=\sum_{\text{$\kappa$ is proper}}\left(\prod_{i=1}^nx_{\kappa(i)}\right).\end{equation*}
\end{definition}
\begin{remark}
    The function is symmetric since $X_G(\boldsymbol{x})=X_G(\sigma(\boldsymbol{x}))$ for any permutation $\sigma$. The function is also homogeneous, since the degree of every term is $n$. Note that $X_G(\boldsymbol{x})$ is independent of the labeling of $G$ and the ordering of the edges of $G$.
\end{remark}
We use different bases to write symmetric functions without needing to use an infinite number of variables. The basis of interest in this paper is the $e$-basis. First, we must define a partition.
\begin{definition}
    Let $\lambda=(\lambda_1,\ldots,\lambda_l)$ be a tuple of positive integers. Then, $\lambda$ is a \emph{partition} of $n$ if $\lambda$ is a weakly decreasing sequence such that $\sum_{i=1}^l\lambda_i=n$. We let $\len(\lambda)=l$ be the length of partition $\lambda$, which is the number of positive integers in the tuple $\lambda$.
\end{definition}
Now, we define an elementary symmetric function.
\begin{definition}
For some partition $\lambda$ of $n$, the \emph{elementary symmetric function} $e_\lambda$ of degree $n$ is
\begin{equation*}
e_\lambda=\prod_{i=1}^{\len(\lambda)}e_{(\lambda_i)},\text{ where } e_{(k)}=\sum_{
    \substack{
        i_1,\ldots,i_k\in\mathbb{N}\\
        i_1<i_2<\cdots<i_k}
}x_{i_1}\cdots x_{i_k}.
\end{equation*}
\end{definition}
Because every chromatic symmetric function is both symmetric and homogeneous, they can be written uniquely as the sum of finitely many elementary symmetric functions \cite[Theorem 7.4.4]{Stanley_Fomin_1999}.
\begin{example}
    If $G$ is the path graph with two vertices, its chromatic symmetric function can be written as
    \begin{equation}\label{eqn:csfP2}
    X_G(\boldsymbol{x})=\sum_{\substack{
    i,j\in\mathbb{N}\\
    i<j}}2x_ix_j=2e_{(2)}.
    \end{equation}
\end{example}
\begin{definition}
    A graph $G$ is \emph{$e$-positive} if the expansion $X_G(\boldsymbol{x})=\sum_{\lambda\in\Lambda_n}c_\lambda e_{\lambda}$ has only non-negative coefficients.
\end{definition}

Now, we define unit interval graphs.
\begin{definition}
    Graph $G$ with labeled vertices $1$ through $n$ is a \emph{natural unit interval graph} if for all $i<j<k$ where $(i,k)\in E(G)$, both $(i,j)\in E(G)$ and $(j,k)\in E(G)$.
\end{definition}
We can finally state the Stanley--Stembridge conjecture.
    \begin{conjecture}[{{\cite[Conjecture~5.5]{STANLEY1993261}}}]\label{cnj:ssc}
    If $G$ is a natural unit interval graph, then it is $e$-positive.
    \end{conjecture}

\section{Forest Triples}\label{sec:ftrip}
Forest triples provide a way to calculate the chromatic symmetric function for a graph. First, we introduce the concept of a composition.
\begin{definition}
    Let $\alpha=(\alpha_1,\ldots,\alpha_l)$ be a tuple of positive integers. Then, $\alpha$ is a \emph{composition} of $n$ (denoted as $\alpha\models n$) if $\sum_{i=1}^l\alpha_i=n$. The length of the composition is denoted as $\len(\alpha)$, which is the number of positive integers, or \emph{parts}, in the composition. We now define a few properties on compositions. The notation $\alpha_l$ always denotes the last part of composition $\alpha$.
    
    For some $\alpha$, we denote $\alpha\setminus\alpha_i=(\alpha_1,\ldots,\alpha_{i-1},\alpha_{i+1},\ldots,\alpha_l)$. For two compositions $\alpha$ and $\beta$, we define $\alpha\cdot \beta=(\alpha_1,\ldots,\alpha_l,\beta_1,\ldots,\beta_l)$.

    We denote a tuple of compositions as $\left(\alpha^{(1)},\ldots,\alpha^{(m)}\right)$, where $m$ is the number of compositions. This allows us to index the tuple for a specific composition and a specific part of the composition. For example, the first part of the second composition can be denoted as $\alpha_1^{(2)}$.
\end{definition}
We can link a composition to a partition by sorting the elements.
\begin{definition}
    If $\alpha$ is a composition of $n$, then \emph{$\sort(\alpha)$} is the partition of $n$ formed by sorting the elements of $\alpha$ in non-increasing order.
\end{definition}
Now, we will introduce the concept of broken circuits. Recall that $G$ is a graph with labeled vertices $1$ through $n$ and ordered edges.
\begin{definition}
    For every cycle $C$ that is a subgraph of $G$, the corresponding \emph{broken circuit} $B\subset G$ is the subgraph of $C$ without the largest edge of $C$.
\end{definition}
\begin{definition}
    A \emph{non-broken} circuit $F$ of graph $G$ is a subgraph of $G$ such that no subgraph of $F$ is a broken circuit.
\end{definition}
Every non-broken circuit is a forest, since if there is a cycle in $F$, then there exists a subgraph of that cycle that is a broken circuit.
\begin{definition}
    A \emph{tree triple} is a tuple $\ttrip=(T,\alpha,r)$ where $T$ is a tree in some non-broken circuit $F$, $\alpha$ is a composition of $|T|$, and $r$ is an integer with $1\leq r\leq \alpha_1$.
\end{definition}
We can finally define what a forest triple is.
\begin{definition}
    A \emph{forest triple} $\forest$ of $G$ is a set of tree triples $\{\ttrip_1,\ldots,\ttrip_m\}$ where $\ttrip_i=(T_i,\alpha^{(i)},r_i)$ and the trees $T_1$ through $T_m$ are the trees of some non-broken circuit $F$ of graph $G$. 

    The type of a forest triple, $\type(\forest)$, is the partition formed by combining and sorting all the elements in every composition in decreasing order. The sign of a forest triple can be defined as
    \begin{equation*}\sign(\forest)=(-1)^{\sum_{i=1}^m\left(\len(\alpha^{(i)})-1\right)}.\end{equation*}
\end{definition}
\begin{definition}
    A tree triple is \emph{unit} if $\len(\alpha)=1$. A \emph{unit} forest triple is one where every tree triple is unit.
\end{definition}
\begin{remark}
    Note if $\forest$ is unit, then $\sign(\forest)=(-1)^{\sum_{i=1}^m(1-1)}=1$. Additionally, if $\ttrip=(T,\alpha,r)$ is unit, then $\alpha_1=|T|$.
\end{remark}

Let $\FT(G)$ be the set of all forest triples of $G$. We can now state the following theorem regarding forest triples.
\begin{theorem}[{{\cite[Theorem~3.4]{tom2024signed}}}]\label{thm:ftt}
    The chromatic symmetric function of any graph $G$ is
    \begin{equation}\label{eq:fte}X_G(\boldsymbol{x})=\sum_{\forest\in\FT(G)}\sign(\forest)\cdot e_{\type(\forest)}.\end{equation}
\end{theorem}
\begin{remark}
Note that given a graph $G$ with labeled vertices, the set $\FT(G)$ depends on how the edges are ordered. However, the formula in Equation~\ref{eq:fte} provides the same result for any arbitrary ordering of the edges, since the chromatic symmetric function of a graph does not depend on any ordering of the edges.
\end{remark}
\begin{example}
Figure~\ref{fig:P2ftrip} shows the four forest triples of the path graph on 2 vertices, $P_2$, and how the sum of $\sign(\forest)\cdot e_{\type(\forest)}$ results in $X_{P_2}(\boldsymbol{x})=2e_{(2)}$ as calculated in Equation~\ref{eqn:csfP2}.
\end{example}
\begin{figure}[!htb]
\centering
\caption{All forest triples for $P_2$.}
\vspace{0.2cm}
    \begin{tikzpicture}[scale=1, font=\footnotesize]

\vertex[black][left]{1}{(0,0)}
\vertex[black][right]{2}{(3,0)}

\node[] at (1.5,0.4) {$\alpha=(1,1), r=1$};
\draw (v1)--(v2);

\vertex[black][left]{1}{(0,1.5)}['];
\vertex[black][right]{2}{(3,1.5)}['];

\node[] at (0,2.1) {$\alpha=(1), r=1$};
\node[] at (3,2.1) {$\alpha=(1), r=1$};

\vertex[black][above]{1}{(7.5,0)}[''];
\vertex[black][above]{2}{(10.5,0)}[''];
\node[] at (9,0.6) {$\alpha=(2), r=2$};
\draw (v''1)--(v''2);

\vertex[black][above]{1}{(7.5,1.5)}['''];
\vertex[black][above]{2}{(10.5,1.5)}['''];

\node[] at (9,2.1) {$\alpha=(2), r=1$};
\draw (v'''1)--(v'''2);
\draw[help lines] (-1.2,-0.5)--(-1.2,2.5)--(13.8,2.5)--(13.8,-0.5)--cycle;
\draw[help lines] (-1.2,2.5)--(-1.2,3)--(13.8,3)--(13.8,2.5);
\draw[help lines] (4.2,-0.5)--(4.2,3);
\node[] at (5.6,2.75) {$\sign(\forest)\cdot e_{\type(\forest)}$};
\node[] at (1.5,2.75) {Forest Triples};
\node[] at (9,2.75) {Forest Triples};
\node[] at (12.4,2.75) {$\sign(\forest)\cdot e_{\type(\forest)}$};
\draw[help lines] (7.0,-0.5)--(7.0,3);
\draw[help lines] (11.0,-0.5)--(11.0,3);
\draw[help lines] (-1.2, 0.9)--(13.8,0.9);
\node[] at (5.6,1.7) {\huge $e_{(1,1)}$};
\node[] at (5.6,0.2) {\huge $-e_{(1,1)}$};
\node[] at (12.4,1.7) {\huge $e_{(2)}$};
\node[] at (12.4,0.2) {\huge $e_{(2)}$};
\end{tikzpicture}
\label{fig:P2ftrip}
\end{figure}
\section{Sign-Reversing Involutions}\label{sec:sri}
Now, we describe a way to prove a graph is $e$-positive using forest triples.
\begin{definition}
    A \emph{sign-reversing involution} is a function $\varphi:\FT(G)\rightarrow\FT(G)$ with the following properties:
    \begin{enumerate}
        \item It is an involution, meaning $\varphi(\varphi(\forest))=\forest$.
        \item It preserves type, so $\type(\forest)=\type(\varphi(\forest))$.
        \item If $\forest\neq \varphi(\forest)$, then $\sign(\forest)\neq \sign(\varphi(\forest))$.
        \item If $\forest=\varphi(\forest)$, then $\sign(\forest)=1$, and we say that $\forest$ is a fixed point.
    \end{enumerate}
\end{definition}
If there exists a sign-reversing involution on $\FT(G)$, then we can pair every forest triple with a negative sign to a non-fixed forest triple with a positive sign, meaning
\begin{equation}\label{eq:sro}
    X_G(\boldsymbol{x})=\sum_{\substack{\forest\in\FT(G)\\\text{$\forest$ is fixed under $\varphi$}}}e_{\type(\forest)}.
\end{equation}

We will look at sign-reversing involutions with an additional property.

\begin{definition}
    For $\forest\in \FT(G)$, let $\ttrip_{\min}=(T_{\min},\alpha^{(\min)},r_{\min})$ reference the unique tree triple where the smallest vertex of $G$ is in $V(T_{\min})$.
\end{definition}

\begin{definition}
    Let $\varphi:\FT(G)\rightarrow\FT(G)$ be a sign-reversing involution. Suppose all fixed points in $\varphi$ are units. If for all $\forest\in\FT(G)$, letting $\forest'=\varphi(\forest)$ with $\ttrip'_{\min}=(T'_{\min},\alpha^{(\min)\prime},r'_{\min})\in\forest'$, we have $\alpha^{(\min)}_1=\alpha^{(\min)\prime}_{1}$ and $r_{\min}=r'_{\min}$, then $\varphi$ is a \emph{first-preserving involution}.
\end{definition}
\begin{definition}
    Let \emph{$\FT^{(i)}(G)$} be the set of forest triples where $\alpha_1^{(\min)}=i$ and $r_{\min}=1$.
    
    Letting $\forest=\{\ttrip_{\min},\ttrip_2,\ldots,\ttrip_m\}$, we define $\type'(\forest)=\sort(\alpha^{(\min)}\setminus \alpha_1^{(\min)}\cdot \alpha^{(2)}\cdots \alpha^{(m)})$. Note that $e_{\type'(\forest)}=e_{\type(\forest)} \mathbin{/} {e_{\alpha_1^{(\min)}}}$.
\end{definition}
If $\FT(G)$ has a first-preserving involution, then there exists sign-reversing involutions on the subsets $\FT^{(i)}(G)$.
\begin{definition}
    For a graph $G$ and integer $i$, we denote
    \begin{equation}\label{eqn:nicecsf}
        X^{(i)}_G(\boldsymbol{x})=\sum_{\substack{
            \forest\in\FT^{(i)}(G)
        }}\sign(\forest)\cdot e_{\type'(\forest)}.
    \end{equation}
    We see $X^{(i)}_G$ is a homogeneous symmetric function of degree $|G|-i$.
\end{definition}
Note that $X^{(i)}_G$ has non-negative coefficients for all integers $i$ if and only if $G$ has a first-preserving involution $\varphi$, letting Equation~\ref{eqn:nicecsf} be rewritten as
\begin{equation}\label{eqn:nicecsfinv}
    X^{(i)}_G(\boldsymbol{x})=\sum_{\substack{
            \forest\in\FT^{(i)}(G)\\
            \text{$\forest$ fixed under $\varphi$}
        }}e_{\type'(\forest)}.
\end{equation}
The chromatic symmetric function can be written as
\begin{equation}\label{eqn:nicecsftocsf}
    X_G(\boldsymbol{x})=\sum_{i=1}^{|G|}e_i\cdot i \cdot X_G^{(i)}(\boldsymbol{x}).
\end{equation}

\begin{example}
    Returning to the forest triples of $G=P_2$ shown in Figure~\ref{fig:P2ftrip}, we have $X_G^{(1)}(\boldsymbol{x})=0$ and $X_G^{(2)}(\boldsymbol{x})=1$.
\end{example}
The first author proved in \cite[Theorem~4.10]{tom2024signed} that all $K\!$-chains, which are cliques joined at single vertices, have a first-preserving involution (for a specific labeling of vertices and ordering of edges). It is conjectured in \cite[Section~5]{tom2024signed} that all natural unit interval graphs with a specific ordering of edges have a first-preserving involution.

To prove that adjacent cycle chains are $e$-positive, we will try to find a first-preserving involution for them. We will first look at cycles.

\section{Cycle Graphs}\label{sec:cgi}
In this section, we find a first-preserving involution on the set of forest triples of cycle graphs.
\begin{definition}
For $a\in\mathbb{N}$, we define $C_a=\left([a],E=\{(1,2)<\cdots<(a-1,a)<(a,1)\}\right)$ as the cycle graph with $a$ vertices.
\end{definition}
Previous papers have found the chromatic symmetric function for cycle graphs. We will prove the same result using forest triples.
\begin{theorem}[{{\cite[Corollary 6.2]{Ellzey2017ADG}}}]\label{thm:cg}
The cycle graph with $a$ vertices, $C_a$, has a chromatic symmetric function equal to
\begin{equation}\label{eqn:cyclecsf}X_{C_a}(\boldsymbol{x})=\sum_{\alpha\models a}e_{\sort(\alpha)}\cdot \alpha_1\cdot (\alpha_1-1)\cdots (\alpha_l-1).\end{equation}
\end{theorem}
\begin{example}
    The graph $C_6$ has
    \begin{equation*}\label{eqn:c6csf}
        X_{C_6}\oset[0.31cm]{\alpha\hspace{0.1cm}=}{(\boldsymbol{x})}=\overbrace{30e_6}^{(6)}+\overbrace{12e_{4,2}}^{(4,2)}+\overbrace{12e_{3,3}}^{(3,3)}+\overbrace{6e_{4,2}}^{(2,4)}+\overbrace{2e_{2,2,2}}^{(2,2,2)}.
    \end{equation*}
\end{example}
    Let $\forest=\{\ttrip_1=(T_1,\alpha^{(1)},r_1),\ldots,\ttrip_m\}\in \FT(C_a)$ be a forest triple with $m$ tree triples. We order the tree triples such that $1\in V(T_m)$ and for $1\leq i<j\leq m-1$, we have $\min(V(T_i))<\min(V(T_j))$.
    
    For each $\ttrip=(T,\alpha,r)\in \forest$, we can find the size of $T$ based on $\alpha$ (since $\alpha\models |T|$). Thus we can identify $\ttrip$ with a tuple $(v,\alpha,r)$, where $v$ is the vertex such that $V(T)=\{(v,v+1),\ldots, (v+|T|-2,v+|T|-1)\}$, with vertices taken mod $a$. We then denote forest triples as
    \begin{equation*}
    \forest=\left\langle(v_1,\alpha^{(1)},r_1),\ldots,(v_m,\alpha^{(m)},r_m)\right\rangle,
    \end{equation*}
    where $T_i$ has edges $\{(v_i,v_i+1),\ldots,(v_{i+1}-2,v_{i+1}-1)\}\!\!\mod a$ and $\ttrip_i=(T_i,\alpha^{(i)},r_i)$. A first-preserving involution preserves $\alpha^{(m)}_1$ and $r_m$, since vertex $T_{\min}=T_m$.
    \begin{examplen}
        An example is given in Subfigure~\ref{subfig:cycexamplebreak:norm}.
    \end{examplen}

    To define a first-preserving involution on $\FT(C_a)$, we break it into disjoint subsets such that every forest triple is in exactly one subset. Then, if there exists a first-preserving involution on each subset, they can be combined together to form a first-preserving involution on $\FT(C_a)$.
    
    Say $\gamma:S_1\rightarrow S_2$ is a bijective function between two disjoint subsets of $\FT(C_a)$, where $\gamma$ preserves $\type(\forest)$ and $\alpha_1^{(m)}$ and $r_m$, but reverses $\sign(\forest)$. Then, we can define a first-preserving involution $\varphi$ on $S_1\cup S_2$, where $\varphi$ either applies $\gamma$ or its inverse. This involution has no fixed points.

\begin{proof}[Proof of Theorem~\ref{thm:cg}]
    We now define the involution in three steps, breaking $\FT(C_a)$ into five subsets.
    \paragraph{Step 1:}
    We define sets $A$ and $B$ with indexed subsets $A_i$ and $B_i$ for $i\in \mathbb{N}$ as
    \begin{gather*}
        A_i=\left\{\forest\in\FT(C_a):m-1\geq i\scolon r_i=1\scolon \len(\alpha^{(1)})=\cdots=\len(\alpha^{(i)})=1\scolon r_1,\ldots,r_{i-1}\geq 2\right\},\\
        B_i=\{\forest\in\FT(C_a):m\geq i\scolon\len(\alpha^{(i)})\geq 2\scolon\len(\alpha^{(1)})=\cdots=\len(\alpha^{(i-1)})=1\scolon r_1,\ldots,r_{i-1}\geq 2\scolon\\ \text{either $m-1\geq i$ or $\alpha_l^{(m)}\leq a-v_m+1$}\}.
    \end{gather*}
    Intuitively, if $\forest\in A_i$ or $\forest\in B_i$, then $i$ is the smallest integer such that either $\len(\alpha^{(i)})\geq 2$ or $r_i=1$. We define function $\tjoin_i:A_i\rightarrow B_i$, where $\tjoin_i(\forest)$ replaces $\ttrip_i$ and $\ttrip_{i+1}$ in $\forest$ with tree triples
    \begin{equation*}
       S=\left(v_{i},\alpha^{(i+1)}\cdot\alpha^{(i)},r_{i+1}\right).
    \end{equation*}
    The inverse map $\tbreak_i:B_i\rightarrow A_i$ replaces $\ttrip_i$ in $\forest$ with
    \begin{equation*}
        S_1=\left(v_i,(\alpha_l^{(i)}),1\right),S_2=\left(v_i+\alpha_l^{(i)},\alpha^{(i)}\setminus \alpha_l^{(i)}, r_i\right).
    \end{equation*}

    Notice that if $\forest\in A_i$, then tree triples $\ttrip_1$ through $\ttrip_{i-1}$ all have $\len(\alpha)=1$ and $r\geq 2$. Those tree triples are untouched in $\tjoin_i(\forest)$, while $\ttrip_i$ and $\ttrip_{i+1}$ are replaced by $S$. Since $S$ has a composition with length at least 2, then $\tjoin_i(\forest)\in B_i$. Applying $\tbreak_i(\tjoin_i(\forest))$ breaks $S$ while preserving the other tree triples, resulting in $\forest$ again. Thus, $\tjoin_i$ and $\tbreak_i$ are inverses, meaning $\tjoin_i$ is a bijection.

    If $\forest\in A_{m-1}$, then $\tjoin_i(\forest)$ replaces $\ttrip_{m-1}$ and $\ttrip_{m}$ with $S$. Since every $\forest$ has $|T_{m-1}|\leq a-v_{m-1}+1$, then after joining $S$ has $\alpha_l^{(m)}\leq a-v_m+1$, explaining that condition. Also note that $\tbreak_i$ preserves $\type(\forest)$, $\alpha_1^{(m)}$, $r_m$, and reverses $\sign(\forest)$. We then use this function to form a first-preserving involution on sets $A\cup B$.

    \begin{examplen}
        Figure~\ref{fig:cycleexamplebreak} shows examples of how $\tjoin_i$ and $\tbreak_i$ work, where we either join the light and dark blue tree triples or break the blue tree triple. In Subfigure~\ref{subfig:cycexamplebreak:twn}, joining $\ttrip_{m-1}$ to $\ttrip_m$ in $\forest_A$ with $m=3$ results in $\forest_B$ having $2=\alpha_l^{(m)}\leq a-v_m+1=2$, hence we can do $\tbreak_2(\forest_B)$ to recover $\forest_A$. Subfigure~\ref{subfig:cycexamplebreak:thr} has $\forest$ with $m=2$ where $3=\alpha_l^{(m)}>a-v_m+1=2$, and trying $\tbreak_2(\forest)$ changes the order of the tree triples and moreover results in $\tbreak_2(\forest)\in A_1$, which is a problem.
    \end{examplen}

    
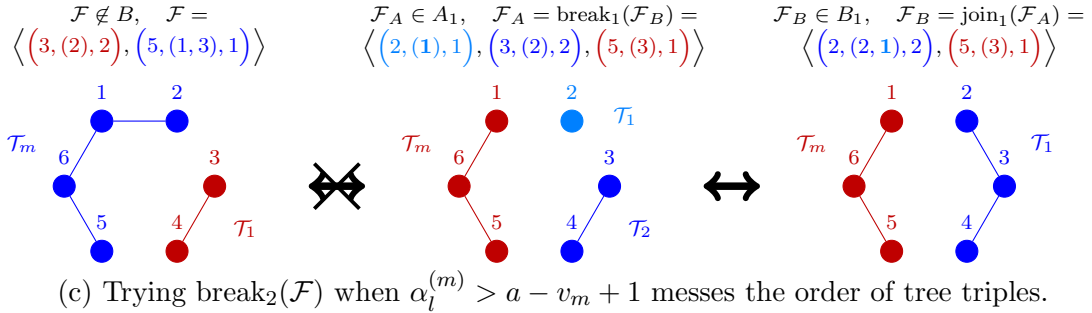
\begin{figure}[!htb]
\caption{Examples of $\tjoin$ and $\tbreak$ for $\forest\in \FT(C_6)$.}
\begin{subfigure}{\textwidth}
\centering
     \begin{tikzpicture}[font=\footnotesize]
\tikzmath{\h0=(4.75-(1+2*cos(60))+3.5;}

\vertex[red]{1}[(\h0,0)]{({cos(60)},{sin(60)})}
\vertex[blue]{2}[(\h0,0)]{({1+cos(60)},{sin(60)})}
\vertex[blue]{3}[(\h0,0)]{({1+2*cos(60)},0)}
\vertex[blue]{4}[(\h0,0)]{({1+cos(60)},{-sin(60)})}
\vertex[blue]{5}[(\h0,0)]{({cos(60)},{-sin(60)})}
\vertex[red]{6}[(\h0,0)]{(0,0)}
\draw[red] (v1)--(v6);
\draw[blue] (v2)--(v3)--(v4)--(v5);
\node[above right=0.2cm and 0.1cm of v3] {$\textcolor{blue}{S=\ttrip_1}$};
\node[above left=0.2cm and 0.1cm of v6] {$\textcolor{red}{\ttrip_m}$};
\node[text width=5cm, align=center] at ({(2*cos(60)+1)/2+\h0},2) {${\forest_B\in B_1,\quad \forest_B=\tjoin_1(\forest_A)=}$\\$\left\langle\textcolor{blue}{\Big(2,(3,\textcolor{lightblue}{\textbf{1}}),3\Big)},\textcolor{red}{\Big(6,(1,1),1\Big)}\right\rangle$};
\draw[<->, line width=0.1cm] (3.75,0) -- (4.5,0);
\tikzmath{\h1=0;}
\vertex[red]{1}[(\h1,0)]{({cos(60)},{sin(60)})}[']
\vertex[lightblue]{2}[(\h1,0)]{({1+cos(60)},{sin(60)})}[']
\vertex[blue]{3}[(\h1,0)]{({1+2*cos(60)},0)}[']
\vertex[blue]{4}[(\h1,0)]{({1+cos(60)},{-sin(60)})}[']
\vertex[blue]{5}[(\h1,0)]{({cos(60)},{-sin(60)})}[']
\vertex[red]{6}[(\h1,0)]{(0,0)}[']
\draw[red] (v'1)--(v'6);
\draw[blue] (v'3)--(v'4)--(v'5);
\node[above right=-0.3cm and 0.3cm of v'2] {$\textcolor{lightblue}{S_1=\ttrip_1}$};
\node[below right=0.2cm and 0cm of v'3] {$\textcolor{blue}{S_2=\ttrip_2}$};
\node[above left=0.2cm and 0.1cm of v'6] {$\textcolor{red}{\ttrip_m}$};
\node[text width=6cm, align=center] at ({(2*cos(60)+1)/2+\h1-0.4},2) {${\forest_A\in A_1,\quad \forest_A=\tbreak_1(\forest_B)=}$\\$\left\langle\textcolor{lightblue}{\Big(2,(\textbf{1}),1\Big)},\textcolor{blue}{\Big(3,(3),3\Big)},\textcolor{red}{\Big(6,(1,1),1\Big)}\right\rangle$};
\end{tikzpicture}
\vspace{2mm}
     \caption{Normal joining and breaking.}
     \label{subfig:cycexamplebreak:norm}
\end{subfigure}
\begin{subfigure}{\textwidth}
\centering
     \begin{tikzpicture}
     \tikzmath{\h0=(4.75-(1+2*cos(60))+3.5;}
\vertex[blue]{1}[(\h0,0)]{({cos(60)},{sin(60)})}
\vertex[blue]{2}[(\h0,0)]{({1+cos(60)},{sin(60)})}
\vertex[red]{3}[(\h0,0)]{({1+2*cos(60)},0)}
\vertex[red]{4}[(\h0,0)]{({1+cos(60)},{-sin(60)})}
\vertex[blue]{5}[(\h0,0)]{({cos(60)},{-sin(60)})}
\vertex[blue]{6}[(\h0,0)]{(0,0)}
\draw[blue] (v2)--(v1)--(v6)--(v5);
\draw[red] (v3)--(v4);
\node[below right=0.2cm and 0cm of v3] {$\textcolor{red}{\ttrip_1}$};
\node[above left=0.2cm and 0.1cm of v6] {$\textcolor{blue}{S=\ttrip_m}$};
\node[text width=5cm, align=center] at ({(2*cos(60)+1)/2+\h0},2) {${\forest_B\in B_2,\quad \forest_B=\tjoin_2(\forest_A)=}$\\$\left\langle\textcolor{red}{\Big(3,(2),2\Big)},\textcolor{blue}{\Big(5,(1,1,\textcolor{lightblue}{\textbf{2}}),1\Big)}\right\rangle$};
\draw[<->, line width=0.1cm] (3.75,0) -- (4.5,0);
\tikzmath{\h1=0;}
\vertex[blue]{1}[(\h1,0)]{({cos(60)},{sin(60)})}[']
\vertex[blue]{2}[(\h1,0)]{({1+cos(60)},{sin(60)})}[']
\vertex[red]{3}[(\h1,0)]{({1+2*cos(60)},0)}[']
\vertex[red]{4}[(\h1,0)]{({1+cos(60)},{-sin(60)})}[']
\vertex[lightblue]{5}[(\h1,0)]{({cos(60)},{-sin(60)})}[']
\vertex[lightblue]{6}[(\h1,0)]{(0,0)}[']
\draw[blue] (v'1)--(v'2);
\draw[red] (v'3)--(v'4);
\draw[lightblue] (v'5)--(v'6);
\node[text width=6cm, align=center] at ({(2*cos(60)+1)/2+\h1-0.3},2) {${\forest_A\in A_2,\quad \forest_A=\tbreak_2(\forest_B)=}$\\$\left\langle\textcolor{red}{\Big(3,(2),2\Big)},\textcolor{lightblue}{\Big(5,(\textbf{2}),1\Big)},\textcolor{blue}{\Big(1,(1,1),1\Big)}\right\rangle$};
\node[above right=-0.3cm and 0.3cm of v'2] {$\textcolor{blue}{S_2=\ttrip_m}$};
\node[below right=0.2cm and 0cm of v'3] {$\textcolor{red}{\ttrip_1}$};
\node[below left=0.2cm and 0cm of v'6] {$\textcolor{lightblue}{S_1=\ttrip_2}$};
\end{tikzpicture}
\vspace{2mm}
     \caption{Joining $\ttrip_{m-1}$ to $\ttrip_m$ and breaking $\ttrip_{m}$.}
     \label{subfig:cycexamplebreak:twn}
\end{subfigure}
\begin{subfigure}{\textwidth}
\centering
    \begin{tikzpicture}[scale=1,font=\scriptsize]
     \tikzmath{\h0=0;}
\vertex[blue]{1}[(\h0,0)]{({cos(60)},{sin(60)})}
\vertex[blue]{2}[(\h0,0)]{({1+cos(60)},{sin(60)})}
\vertex[red]{3}[(\h0,0)]{({1+2*cos(60)},0)}
\vertex[red]{4}[(\h0,0)]{({1+cos(60)},{-sin(60)})}
\vertex[blue]{5}[(\h0,0)]{({cos(60)},{-sin(60)})}
\vertex[blue]{6}[(\h0,0)]{(0,0)}
\draw[blue] (v2)--(v1)--(v6)--(v5);
\draw[red] (v3)--(v4);
\node[below right=0.2cm and 0cm of v3] {$\textcolor{red}{\ttrip_1}$};
\node[above left=0.2cm and 0.1cm of v6] {$\textcolor{blue}{\ttrip_m}$};
\node[text width=5cm, align=center] at ({(2*cos(60)+1)/2+\h0},2) {${\forest\not\in B,\quad \forest=}$\\$\left\langle\textcolor{red}{\Big(3,(2),2\Big)},\textcolor{blue}{\Big(5,(1,3),1\Big)}\right\rangle$};
\draw[<->, line width=0.1cm] (3.25,0) -- (4,0);
\draw[line width=0.5mm] ({3.625+0.3},0.3)--({3.625-0.3},-0.3) ({3.625+0.3},-0.3)--({3.625-0.3},0.3);
\tikzmath{\h1=(3.25-(1+2*cos(60))+4;}
\vertex[red]{1}[(\h1,0)]{({cos(60)},{sin(60)})}[']
\vertex[lightblue]{2}[(\h1,0)]{({1+cos(60)},{sin(60)})}[']
\vertex[blue]{3}[(\h1,0)]{({1+2*cos(60)},0)}[']
\vertex[blue]{4}[(\h1,0)]{({1+cos(60)},{-sin(60)})}[']
\vertex[red]{5}[(\h1,0)]{({cos(60)},{-sin(60)})}[']
\vertex[red]{6}[(\h1,0)]{(0,0)}[']
\draw[blue] (v'3)--(v'4);
\draw[red] (v'5)--(v'6)--(v'1);
\node[above right=-0.3cm and 0.3cm of v'2] {$\textcolor{lightblue}{\ttrip_1}$};
\node[below right=0.2cm and 0cm of v'3] {$\textcolor{blue}{\ttrip_2}$};
\node[above left=0.2cm and 0.1cm of v'6] {$\textcolor{red}{\ttrip_m}$};
\node[text width=6cm, align=center] at ({(2*cos(60)+1)/2+\h1},2) {${\forest_A\in A_1,\quad \forest_A=\tbreak_1(\forest_B)=}$\\$\left\langle\textcolor{lightblue}{\Big(2,(\textbf{1}),1\Big)},\textcolor{blue}{\Big(3,(2),2\Big)},\textcolor{red}{\Big(5,(3),1\Big)}\right\rangle$};
\draw[<->, line width=0.1cm] ({3.25+\h1},0) -- ({4+\h1},0);
\vertex[red]{1}[(2*\h1,0)]{({cos(60)},{sin(60)})}[']
\vertex[blue]{2}[(2*\h1,0)]{({1+cos(60)},{sin(60)})}[']
\vertex[blue]{3}[(2*\h1,0)]{({1+2*cos(60)},0)}[']
\vertex[blue]{4}[(2*\h1,0)]{({1+cos(60)},{-sin(60)})}[']
\vertex[red]{5}[(2*\h1,0)]{({cos(60)},{-sin(60)})}[']
\vertex[red]{6}[(2*\h1,0)]{(0,0)}[']
\draw[blue] (v'2)--(v'3)--(v'4);
\draw[red] (v'5)--(v'6)--(v'1);
\node[above right=0.2cm and 0.1cm of v'3] {$\textcolor{blue}{\ttrip_1}$};
\node[above left=0.2cm and 0.1cm of v'6] {$\textcolor{red}{\ttrip_m}$};
\node[text width=6cm, align=center] at ({(2*cos(60)+1)/2+2*\h1},2) {${\forest_B\in B_1,\quad \forest_B=\tjoin_1(\forest_A)=}$\\$\left\langle\textcolor{blue}{\Big(2,(2,\textcolor{lightblue}{\textbf{1}}),2\Big)},\textcolor{red}{\Big(5,(3),1\Big)}\right\rangle$};
\end{tikzpicture}
     \caption{Trying $\tbreak_2(\forest)$ when $\alpha_l^{(m)}>a-v_m+1$ messes the order of tree triples.}
     \label{subfig:cycexamplebreak:thr}
\end{subfigure}
\label{fig:cycleexamplebreak}
\end{figure}
    \paragraph{Step 2:}
    To take care of the forest triples like the one on the left in Subfigure~\ref{subfig:cycexamplebreak:thr}, we will rotate the entire graph until we can break $\ttrip_m$ while preserving $\alpha_1^{(m)}$ and $r_m$.

    We define sets $C,D$ with indexed subsets $C_i$ and $D_i$ for $i\in\mathbb{N}$ as
    \begin{gather*}
        C_i=\{(a,1)\not\in E(T_m)\scolon \len(\alpha^{(1)})=\cdots=\len(\alpha^{(m-1)})=1\scolon r_1,\ldots,r_{m-1}\geq 2\scolon r_{m-1}-1=i\},\\
        D_i=\{(a,1)\in E(T_m)\scolon \len(\alpha^{(1)})=\cdots=\len(\alpha^{(m-1)})=1\scolon r_1,\ldots, r_{m-1}\geq 2\scolon\\ \len(\alpha^{(m)})\geq 2\scolon \alpha_l^{(m)}\geq a-v_m+2\scolon \alpha_l^{(m)}-a+v_m-1=i\}.
    \end{gather*}
    In other words, we now consider $\forest\in \FT(C_a)$ where $\forest\not\in A\cup B$. If edge $(a,1)\not\in E(T_m)$, then we put $\forest\in C_i$, where $i=r_{m-1}-1$. If edge $(a,1)\in E(T_m)$ and $\len(\alpha^{(m)})\geq 2$, then we put $\forest\in D_i$ where $i=\alpha_l^{(m)}-a+v_m-1$.

    We define function $\rotatejoin_i:C_i\rightarrow D_i$ as
    \begin{multline*}
        \rotatejoin_i(\forest)=\Big\langle(v_1+i,\alpha^{(1)},r_1),\ldots,(v_{m-2}+i,\alpha^{(m-2)},r_{m-2}),\\(v_{m-1}+i,\alpha^{(m)}\cdot\alpha^{(m-1)},r_m)\Big\rangle.
    \end{multline*}
    Say $\forest_C\in C_i$, and $\forest_D=\rotatejoin_i(\forest_C)$. Note that in $\forest_C$, we have $\alpha_1^{(m-1)}=a-v_{m-1}+1$. Then,
    \begin{align*}
        \text{$(v_m)$ in $\forest_D$}&=\text{$(v_{m-1}+i)$ in $\forest_C$, so}\\
        \text{$(\alpha_l^{(m)}-a+v_m-1)$ in $\forest_D$}&=\text{$(\alpha_1^{(m-1)}-a+v_{m-1}+i-1)$ in $\forest_C$}\\
        &=i,
    \end{align*}
    So $\forest_D\in D_i$. The inverse map $\rotatebreak_i:D_i\rightarrow C_i$ is defined as
    \begin{multline*}
        \rotatebreak_i(\forest)=\Big\langle (v_1-i, \alpha^{(1)},r_1),\ldots, (v_{m-1}-i,\alpha^{(1)},r_m),\\ (v_m-i, (\alpha_l^{(m)}), i+1), (1, \alpha^{(m)}\setminus \alpha_l^{(m)}, r_m)\Big\rangle.
    \end{multline*}
    We turn this bijection into a first-preserving involution on set $C\cup D$.
    \begin{examplen}
        Figure~\ref{fig:cycCDinvolution} shows an example of $\rotatejoin_1$ and $\rotatebreak_1$, using the same forest triple in Subfigure~\ref{subfig:cycexamplebreak:thr} that could not be broken with normal $\tbreak$.
    \end{examplen}

    \begin{figure}[!htb]
\caption{Example of $\rotatejoin$ and $\rotatebreak$.}
\centering
     \begin{tikzpicture}
\tikzmath{\h0=5;}
\vertex[blue]{1}[(\h0,0)]{({cos(60)},{sin(60)})}
\vertex[blue]{2}[(\h0,0)]{({1+cos(60)},{sin(60)})}
\vertex[red]{3}[(\h0,0)]{({1+2*cos(60)},0)}
\vertex[red]{4}[(\h0,0)]{({1+cos(60)},{-sin(60)})}
\vertex[blue]{5}[(\h0,0)]{({cos(60)},{-sin(60)})}
\vertex[blue]{6}[(\h0,0)]{(0,0)}
\draw[blue] (v5)--(v6) (v2)--(v1)--(v6);
\draw[red] (v3)--(v4);
\node[below right=0.2cm and 0cm of v3] {$\textcolor{red}{\ttrip_1}$};
\node[above left=0.2cm and 0.1cm of v6] {$\textcolor{blue}{\ttrip_m}$};
\node[align=center, text width=6cm] at ({(2*cos(60)+1)/2+\h0},2) {${\forest_D\in D_1,\quad\forest_D=\rotatejoin_1(\forest_C)=}$\\$\left\langle\textcolor{red}{\Big(3,(2),2\Big)},\textcolor{blue}{\Big(5,(1,\textcolor{lightblue}{\textbf{3}}),1\Big)}\right\rangle$};
\tikzmath{\hi=-5;}
\draw[<->, line width=0.1cm] ({\h0+\hi/2+0.5},0) -- ({\h0+\hi/2+1+0.5},0);
\vertex[blue]{1}[({\h0+\hi},0)]{({cos(60)},{sin(60)})}[']
\vertex[red]{2}[({\h0+\hi},0)]{({1+cos(60)},{sin(60)})}[']
\vertex[red]{3}[({\h0+\hi},0)]{({1+2*cos(60)},0)}[']
\vertex[blue]{4}[({\h0+\hi},0)]{({1+cos(60)},{-sin(60)})}[']
\vertex[blue]{5}[({\h0+\hi},0)]{({cos(60)},{-sin(60)})}[']
\vertex[blue]{6}[({\h0+\hi},0)]{(0,0)}[']
\draw[red] (v'2)--(v'3);
\draw[blue] (v'4)--(v'5)--(v'6)--(v'1);
\node[align=center, text width=5cm] at ({(2*cos(60)+1)/2+\h0+\hi},2) {Intermediate Step};
\draw[<->, line width=0.1cm] ({\h0+\hi*3/2+0.5},0) -- ({\h0+\hi*3/2+1+0.5},0);
\vertex[blue]{1}[({\h0+2*\hi},0)]{({cos(60)},{sin(60)})}[2']
\vertex[red]{2}[({\h0+2*\hi},0)]{({1+cos(60)},{sin(60)})}[2']
\vertex[red]{3}[({\h0+2*\hi},0)]{({1+2*cos(60)},0)}[2']
\vertex[lightblue]{4}[({\h0+2*\hi},0)]{({1+cos(60)},{-sin(60)})}[2']
\vertex[lightblue]{5}[({\h0+2*\hi},0)]{({cos(60)},{-sin(60)})}[2']
\vertex[lightblue]{6}[({\h0+2*\hi},0)]{(0,0)}[2']
\draw[red] (v2'2)--(v2'3);
\draw[lightblue] (v2'4)--(v2'5)--(v2'6);
\node[align=center, text width=6cm] at ({(2*cos(60)+1)/2+\h0+2*\hi},2) {${\forest_C\in C_1,\quad \forest_C=\rotatebreak_1(\forest_D)=}$\\$\left\langle\textcolor{red}{\Big(2,(2),2\Big)},\textcolor{lightblue}{\Big(4,(\textbf{3}),2\Big)},\textcolor{blue}{\Big(1,(1),1\Big)}\right\rangle$};
\node[above right=0.2cm and 0.1cm of v2'3] {$\textcolor{red}{\ttrip_1}$};
\node[below left=0.2cm and 0.0cm of v2'6] {$\textcolor{lightblue}{\ttrip_2}$};
\node[above left=-0.3cm and 0.15cm of v2'1] {$\textcolor{blue}{\ttrip_m}$};
\end{tikzpicture}
\label{fig:cycCDinvolution}
\end{figure}
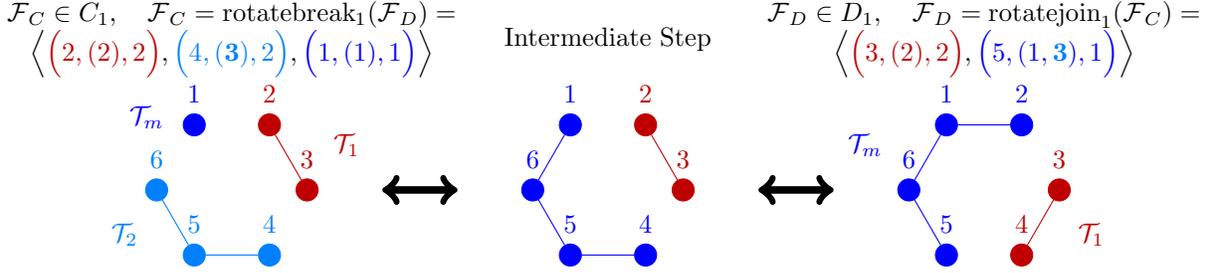
    
    \paragraph{Step 3:}
    The remaining $\forest\in \FT(C_a)$ have the property that $\ttrip_1$ through $\ttrip_{m}$ have $\len(\alpha)=1$, while $\ttrip_1$ through $\ttrip_{m-1}$ have $r\geq 2$. Additionally, edge $(a,1)\in E(T_m)$. We place these $\forest$ in set $E$.

    Notice that all $\forest\in E$ are unit, since every tree triple has $\len(\alpha)=1$. Thus, $\sign(\forest)$ is always positive, meaning the identity function is a first-preserving involution on set $E$. 
    \paragraph{Counting Fixed Points:}
    Combining the involutions on $A\cup B,C\cup D,E$ gives a first-preserving involution $\varphi:\FT(C_a)\rightarrow \FT(C_a)$, where all fixed points are in subset $E$.

    We use Equation~\ref{eq:sro} to find the chromatic symmetric function of a cycle graph. For all $\forest\in E$, we define $\comp(\forest)=\alpha^{(m)}\cdot \alpha^{(1)}\cdots \alpha^{(m-1)}$. Now, for some $\beta\models a$, we want to find the number of $\forest\in E$ where $\comp(\forest)=\beta$. Since edge $(a,1)\in T_m$, there are $|T_m|-1=\beta_1-1$ ways to place $T_m$, and one way to place the remaining $T_1$ through $T_{m-1}$. There are $|T_i|-1$ possible $r$ values for $\ttrip_1$ through $\ttrip_{m-1}$, and $|T_m|$ possible $r$ values for $\ttrip_m$. Thus, there are $\beta_1\cdot(\beta_1-1)\cdots(\beta_l-1)$ possible forest triples where $\comp(\forest)=\beta$.

    If $\comp(\forest)=\beta$, then $\type(\forest)=\sort(\beta)$. Thus, we get
    \[X_{C_a}(\boldsymbol{x})=\sum_{\beta \models a}e_{\sort(\beta)}\cdot \beta_1\cdot (\beta_1-1)\cdots(\beta_l-1),\]
    which is equivalent to Equation~\ref{eqn:cyclecsf}.

    We can also calculate $X^{(i)}_{C_a}(\boldsymbol{x})$ using similar logic, counting the number of forest triples where $\comp(\forest)=\beta$ and $\beta_1=i$ and $r_{\min}=1$, getting that
    \begin{equation}\label{eqn:cyclecsfspec}
        X_{C_a}^{(i)}(\boldsymbol{x})=\sum_{\alpha\models a-i}e_{\sort(\alpha)}\cdot(i-1)\cdot (\alpha_1-1)\cdots(\alpha_l-1).
    \end{equation}
\end{proof}

\section{Adjacent Cycle Chains}\label{sec:cycchain}
In this section, we will prove that adding a cycle to a graph preserves the existence of a first-preserving involution. From this, we get that adjacent cycle chains are $e$-positive. First, we must define what adding a cycle to a graph means.

\begin{definition}
   Let $G_1$ and $G_2$ be two graphs with $k_1$ and $k_2$ labeled vertices respectively. Both graphs have some fixed ordering of edges. Following the notation in \cite{Gebhard2001}, we define
   \[G_1+G_2=([k_1+k_2-1],E=E(G_1)\cup \{(i+k_1-1,j+k_1-1) \mid (i,j)\in E(G_2)\}),\]
    preserving the ordering of the edges in $G_1$ and $G_2$ and making the edges of $G_1$ smaller than the edges of $G_2$.
\end{definition}
\begin{example}
    Figure~\ref{fig:addingcycle} shows a labeled graph $G'$ as well as the new labeled graph $G=C_6+G'$. Since $G'$ is a $K\!$-chain, it has a first-preserving involution, so $G$ should also have a first-preserving involution.
\end{example}
\begin{figure}[!htb]
\centering
\caption{The graph $G'$ next to $C_6+G'$.}
\begin{tikzpicture}
\vertex{1}{(0,0)};
\vertex{2}{({sqrt(2)/2},{sqrt(2)/2})};
\vertex[black][below]{3}{({sqrt(2)/2},{-sqrt(2)/2})};
\vertex{4}{({sqrt(2)},0)};
\vertex{5}{({sqrt(2)+cos(30)},{sin(30)})};
\vertex[black][below]{6}{({sqrt(2)+cos(30)},-{sin(30)})};
\draw (v1)--(v2)--(v3)--(v4)--(v5)--(v6)--(v4)--(v1)--(v3) (v2)--(v4);
\end{tikzpicture}
\hspace{1cm}
\begin{tikzpicture}
\vertex{5}{({-cos(60)},{sin(60)})}
\vertex{4}{({-1-cos(60)},{sin(60)})}
\vertex{3}{({-1-2*cos(60)},0)}
\vertex{2}{({-1-cos(60)},{-sin(60)})}
\vertex{1}{({-cos(60)},{-sin(60)})}
\vertex{6}{(0,0)};
\vertex{7}{({sqrt(2)/2},{sqrt(2)/2})};
\vertex[black][below]{8}{({sqrt(2)/2},{-sqrt(2)/2})};
\vertex{9}{({sqrt(2)},0)};
\vertex{10}{({sqrt(2)+cos(30)},{sin(30)})};
\vertex[black][below]{11}{({sqrt(2)+cos(30)},-{sin(30)})};
\draw (v1)--(v2)--(v3)--(v4)--(v5)--(v6)--(v7)--(v8)--(v9)--(v10)--(v11)--(v9)--(v6)--(v8) (v7)--(v9) (v1)--(v6);
\end{tikzpicture}
\label{fig:addingcycle}
\end{figure}
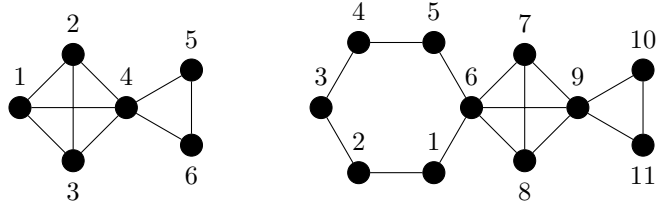
Before jumping to cycle chains, let us first look at cycles connected to trees.
\begin{definition}
    Let $U_{k}$ be a tree with $k$ vertices for some positive integer $k$. For every $\forest\in\FT(C_a+U_{k})$, let $\ttrip=(T,\alpha,r)$ be the tree triple where vertex $a\in V(T)$. We define $\FT'(C_a+U_{k})\subseteq \FT(C_a+U_{k})$ to be the set of all forest triples $\forest$ where $E(U_k)\subseteq E(T)$ and $\alpha_l\geq k$.
\end{definition}
\begin{definition}
    We define $\FT'^{(i)}(C_a+U_k)\subseteq \FT'(C_a+U_k)$ as the subset of forest triples where $\alpha^{(\min)}_1=i$ and $r_{\min}=1$.
\end{definition}
In Section~\ref{sec:cpti}, we will prove the following lemma that the subset $\FT'(C_a+U_k)$ always has an involution.
\begin{lemma}\label{lem:mutatedtadpole}
    The subset $\FT'(C_a+U_k)$ has a first-preserving involution $\varphi$. Additionally,
    \begin{multline}\label{eqn:bak}
        B_{a,k}(\boldsymbol{x}):=\sum_{\mathclap{\substack{\forest\in\FT'(C_a+U_k)}}}\quad\sign(\forest)\cdot e_{\type(\forest)}=\\
        \sum_{\mathclap{\substack{
            \alpha\models a+k,\;
            \len(\alpha)\geq 2\\
            \alpha_1\leq k,\;
            \alpha_2\geq k
        }}}\quad(\alpha_2-\alpha_1+1)(\alpha_1+\alpha_2-k-1)\cdot(\alpha_3-1)\cdots(\alpha_l-1)\cdot e_{\sort(\alpha_1-1,\alpha_2,\ldots,\alpha_l)}.
    \end{multline}
    Moreover, we can break this function into $e$-positive pieces with
    \begin{multline}\label{eqn:specBak}
        B_{a,k}^{(i)}(\boldsymbol{x}):=\sum_{\mathclap{\substack{\forest\in\FT'^{(i)}(C_a+U_k)}}}\quad\sign(\forest)\cdot e_{\type'(\forest)}=\\
        \begin{cases}
            \displaystyle{\smashoperator[r]{\sum_{\substack{\alpha\models a+k-i-1\\\alpha_1\geq k}}}}\quad(k-i)\cdot (\alpha_2-1)\cdots (\alpha_l-1)\cdot e_{\sort(\alpha)}\hspace{0cm},&\text{\normalfont if }i\leq k-1\\[30pt]
            \hspace{2mm}\displaystyle{\smashoperator[r]{\sum_{\substack{\alpha\models a+k-i\\\alpha_1\leq k}}}}\quad(i-k)\cdot(\alpha_2-1)\cdots (\alpha_l-1)\cdot e_{\sort(\alpha_1-1,\alpha\setminus\alpha_1)},&\text{\normalfont if }i\geq k
        \end{cases},
    \end{multline}
    where similar to Equation~\ref{eqn:nicecsftocsf}, we have
    \begin{equation}\label{eqn:nicebftobf}
        B_{a,k}(\boldsymbol{x})=\sum_{i=1}^{a+k-1}e_i\cdot i\cdot B_{a,k}^{(i)}(\boldsymbol{x}).
    \end{equation}
\end{lemma}
Proving this lemma will involve constructing a first-preserving involution for the set $\FT'(C_a+U_k)$, in similar manner to Section~\ref{sec:cgi}. Then, we count the fixed points and derive the formulas above. It turns out that the shape of $U_k$ does not matter, as $B_{a,k}(\boldsymbol{x})$ only depends on $a$ and $k$.

Now, we turn to adding a cycle to a graph.
\begin{theorem}\label{thm:c+g}
    If $G'$ is a graph where $\FT(G')$ has a first-preserving involution, then the graph $G=C_a+G'$ also has a first-preserving involution on $\FT(G)$, and the chromatic symmetric function for $G$ can be written as
    \begin{equation}\label{eqn:csfC_a+G'}
        X_{G}(\boldsymbol{x})=\sum_{k=1}^{|G'|}X^{(k)}_{G'}(\boldsymbol{x})\cdot B_{a,k}(\boldsymbol{x}).
    \end{equation}
    The chromatic symmetric function can be broken into pieces
    \begin{equation}\label{eqn:csfC_a+G'spec}
        X_G^{(i)}(\boldsymbol{x})=\sum_{k=1}^{|G'|}X^{(k)}_{G'}(\boldsymbol{x})\cdot B_{a,k}^{(i)}(\boldsymbol{x}).
    \end{equation}
\end{theorem}
Before proving this theorem, we must introduce a few definitions.

Let $\forest=\left\{\ttrip_1,\ldots,\ttrip_i,\ldots,\ttrip_m\right\}\in\FT(G)$, with tree triples ordered based on their largest vertex, so if $1\leq p<q\leq m$, then $\max(V(T_p))<\max(V(T_q))$. Let $\ttrip_i$ be the tree triple with vertex $a$. Note that tree triples $\ttrip_1$ through $\ttrip_{i-1}$ only contain vertices of the cycle (vertices 1 through $a$), while $\ttrip_{i+1}$ through $\ttrip_m$ only contain vertices of graph $G'$.

The idea behind the new involution will be to ``restrict" a forest triple to $G'$, apply the first-preserving involution on $G'$ to that part, and then recombine the parts. Thus, we want to define how to restrict a forest triple $\forest$ to a specific graph. To do this, we first define how to restrict a tree to a graph.
\begin{definition}
For subtree $T$ of graph $G$, we define \emph{$T \vert_{G'}$}, or \emph{$T$ restricted to $G'$}, as the induced subgraph of $T$ restricted to vertices in $G'$. Similarly, we define \emph{$T\vert_{C_a}$} as the induced subgraph of $T$ restricted to vertices in $C_a$.
\end{definition}
\begin{example}
    The blue tree $T_2$ in Figure~\ref{fig:restrictingtog} on the left contains vertices in both the cycle and the main graph. Then, $T_2\vert_{G'}$ is denoted as the portion of the tree contained in $G'$, which is the light blue tree in the right hand side of the diagram. Similarly, $T_2\vert_{C_a}$ is the dark blue tree in the right hand side of the diagram.
\end{example}
\begin{definition}
    Let $T_1$ be a subtree of $C_a$ and $T_2$ be a subtree of $G'$, where vertex $a\in V(T_1)$ and $a\in V(T_2)$. Then, we define $T_1+T_2=T'$ to be a new tree with vertices $V(T_1)\cup V(T_2)$ and edges $E(T_1)\cup E(T_2)$. Note that $T\vert_{C_a} + T\vert_{G'}=T$.
\end{definition}
\begin{figure}[!htb]
\centering
\caption{A forest triple in $\FT(G)$ being broken into two parts.}
     \begin{tikzpicture}
\vertex[blue]{5}{({-cos(60)},{sin(60)})}
\vertex[blue]{4}{({-1-cos(60)},{sin(60)})}
\vertex[red]{3}{({-1-2*cos(60)},0)}
\vertex[red]{2}{({-1-cos(60)},{-sin(60)})}
\vertex[blue]{1}{({-cos(60)},{-sin(60)})}
\vertex[blue]{6}{(0,0)};
\vertex[darkgreen]{7}{({sqrt(2)/2},{sqrt(2)/2})};
\vertex[darkgreen][below]{8}{({sqrt(2)/2},{-sqrt(2)/2})};
\vertex[blue]{9}{({sqrt(2)},0)};
\vertex[blue]{10}{({sqrt(2)+cos(30)},{sin(30)})};
\vertex[purple][below]{11}{({sqrt(2)+cos(30)},-{sin(30)})};
\draw[blue] (v1)--(v6)--(v5)--(v4) (v6)--(v9)--(v10);
\draw[red] (v2)--(v3);
\draw[darkgreen] (v7)--(v8);
\node[text width=5cm, align=center] at ({(-1-2*cos(60)+sqrt(2)+cos(30))/2},2.1) {$\forest=\Big\langle\textcolor{red}{\Big(T_1,(2),1\Big)},\textcolor{blue}{\Big(T_2,(5,1),5\Big)},$\\$\textcolor{darkgreen}{\Big(T_3,(1,1),1\Big)},\textcolor{purple}{\Big(T_4,(1),1\Big)}\Big\rangle$};
\draw[->, line width=0.1cm] (3,0) -- (3.75,0);
\end{tikzpicture}
\begin{tikzpicture}
\vertex[blue]{5}{({-cos(60)},{sin(60)})}
\vertex[blue]{4}{({-1-cos(60)},{sin(60)})}
\vertex[red]{3}{({-1-2*cos(60)},0)}
\vertex[red]{2}{({-1-cos(60)},{-sin(60)})}
\vertex[blue]{1}{({-cos(60)},{-sin(60)})}
\vertex[blue]{6}{(0,0)};
\draw[line width=0.1cm] (0.5,0)--(1,0) (0.75,-0.25)--(0.75,0.25);
\vertex[lightblue]{6}{(1.5,0)}['];
\vertex[darkgreen]{7}{({1.5+sqrt(2)/2},{sqrt(2)/2})};
\vertex[darkgreen][below]{8}{({1.5+sqrt(2)/2},{-sqrt(2)/2})};
\vertex[lightblue]{9}{({1.5+sqrt(2)},0)};
\vertex[lightblue]{10}{({1.5+sqrt(2)+cos(30)},{sin(30)})};
\vertex[purple][below]{11}{({1.5+sqrt(2)+cos(30)},-{sin(30)})};
\draw[blue] (v1)--(v6)--(v5)--(v4);
\draw[lightblue] (v'6)--(v9)--(v10);
\draw[red] (v2)--(v3);
\draw[darkgreen] (v7)--(v8);
\node[text width=3cm, align=center] at ({(-1-2*cos(60)+0)/2},2.3) {${\forest\vert_{C_a}=}$\\$\Big\langle\textcolor{red}{\Big(T_1,(2),1\Big)},$\\$\textcolor{blue}{\Big(T_2|_{C_a},(4),5\Big)\Big\rangle}$};
\node[text width=5cm, align=center] at ({(1.5+1.5+sqrt(2)+cos(30))/2},2.6) {${\forest\vert_{G'}=}$\\$\Big\langle\textcolor{lightblue}{\Big(T_2|_{G'},(2,1),1\Big)},$\\$\textcolor{darkgreen}{\Big(T_3,(1,1),1\Big)},$\\$\textcolor{purple}{\Big(T_4,(1),1\Big)}\Big\rangle$};
\end{tikzpicture}
\label{fig:restrictingtog}
\end{figure}
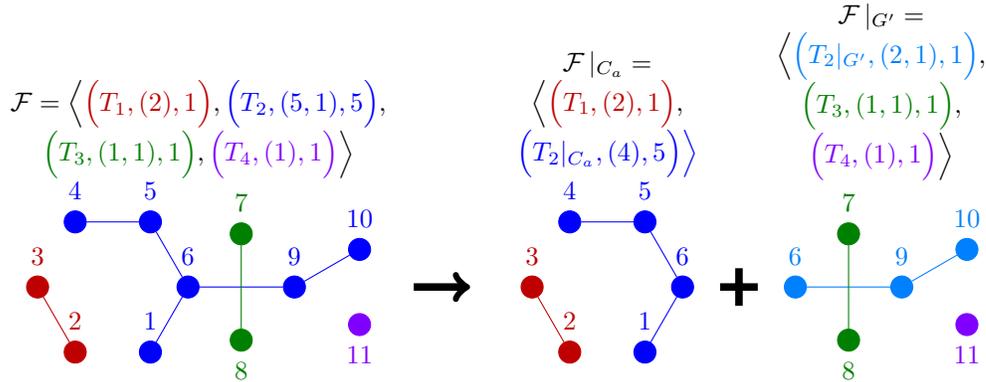
We must also define how to break up a composition.
\begin{definition}
    Let $\alpha=\left(\alpha_1,\ldots,\alpha_m\right)\models n$ be a composition and let $1\leq j\leq n$ be an integer. We define \emph{$\last(\alpha,j)$} as the unique composition of $j$ of the form $(k,\alpha_{i+1},\ldots,\alpha_m)$,
    such that $1\leq k\leq \alpha_i$. We can also define \emph{$\first(\alpha,j)$} as the unique composition of $j$ of the form $(\alpha_1,\ldots,\alpha_{i-1},k)$,
    such that $1\leq k\leq \alpha_i$.
\end{definition}
\begin{example}
    Let $\alpha=(3,2,8,4)\models 17$, then $\last(\alpha,4)=(4)$ and $\last(\alpha,11)=(7,4)$. Similarly, $\first(\alpha,5)=(3,2)$ and $\first(\alpha,15)=(3,2,8,2)$.
\end{example}
\begin{definition}
    Let $\forest=\{\ttrip_1,\ldots,\ttrip_i,\ldots,\ttrip_m\}\in\FT(G)$, with the tree triples ordered as previously stated and with vertex $a\in V(T_i)$. We now define \emph{$\forest\vert _{G'}$}, or \emph{$\forest$ restricted to $G'$}, as
    \begin{equation*}
        \forest\vert_{G'}=\left\{(T_i\vert_{G'},\last\!\left(\alpha^{(i)},\bigl|\,T_i\vert_{G'}\,\bigr|\right), 1)\right\} \cup \left\{\ttrip_{i+1},\ldots,\ttrip_m\right\}.
    \end{equation*}
    We also define $\forest\vert_{C_a}$ as
    \begin{equation*}
        \forest\vert_{C_a}=\left\{\ttrip_1\ldots,\ttrip_{i-1}\right\}\cup \left\{(T_i\vert_{C_a},\first\!\left(\alpha^{(i)},\bigl|\,T_i\vert_{C_a}\,\bigr|\right), r_i)\right\}.
    \end{equation*}
\end{definition}
\begin{remark}
    Note that $\forest\vert_{G'}$ is a forest triple in $\FT(G')$. Similarly, $\forest\vert_{C_a}$ is almost always a forest triple of $C_a$. However, it is possible for $r_i> \alpha_1^{(i)}$ in $\forest\vert_{C_a}$. This does not matter, since when we reattach $\varphi(\forest\vert_{G'})$ to $\forest\vert_{C_a}$, we will still end up with a $\forest\in\FT(G)$. This will be elaborated on in the following paragraphs.
\end{remark}
\begin{definition}
    We define an \emph{almost forest triple} $\forest$ of cycle $C_a$ as an object $\left\{\ttrip_1,\ldots,\ttrip_m\right\}$ (with the largest vertex in $V(T_m)$) with all the properties of a forest triple, except that if $\len(\alpha^{(m)})=1$ then it is possible for $r_m> \alpha_1^{(m)}$. All forest triples of $C_a$ are also almost forest triples.
\end{definition}
\begin{example}
    Looking at Figure~\ref{fig:restrictingtog}, we can break the forest triple $\forest$ into two separate objects, $\forest\vert_{C_a}$ and $\forest\vert_{G'}$. Note that $\forest\vert_{C_a}$ is not actually a forest triple because $r_2>\alpha_1^{(2)}$, meaning it is an almost forest triple. Meanwhile, $\forest\vert_{G'}$ is always a forest triple of graph $G'$.
\end{example}
We now define how to combine an almost forest triple of $C_a$ and a forest triple of $G'$.

\begin{definition}
    Let $\forest_1=\left\{\ttrip_1,\ldots,\ttrip_i\right\}$ be an almost forest triple of $C_a$, where vertex $a\in T_i$. Let $\forest_2=\left\{\ttrip_{i+1},\ldots,\ttrip_m\right\}$ be a forest triple of $G'$, with $\ttrip_{\min}=\ttrip_{i+1}$ and $r_{i+1}=1$. Then, if $r_i\leq \alpha_1^{(i)}$, or both $\len(\alpha_1^{(i)})=1$ and $r_i\leq \alpha_1^{(i)}+\alpha_1^{(i+1)}$, we define $\forest_1+\forest_2$ as
    \begin{equation*}
        \left\{\ttrip_1,\ldots,\ttrip_{i-1}\right\}\cup \left\{\left(T_i+T_{i+1},(\alpha_1^{(i)},\ldots,\alpha_l^{(i)}+\alpha_1^{(i+1)},\ldots,\alpha_l^{(i+1)}),r_i\right)\right\}\cup\left\{\ttrip_{i+2},\ldots,\ttrip_m\right\}.
    \end{equation*}
\end{definition}
\begin{remark}
    Note that for any forest triple $\forest$ of $G$, then $\forest\vert_{C_a} + \forest\vert_{G'}=\forest$.
\end{remark}
We can now prove Theorem~\ref{thm:c+g}.
\begin{proof}[Proof of Theorem~\ref{thm:c+g}]
Let $\varphi':\FT(G')\rightarrow\FT(G')$ be a first-preserving involution on graph $G'$. We want to find a first-preserving involution $\varphi:\FT(G)\rightarrow\FT(G)$. For a forest triple $\forest$, we denote tree triples $\ttrip_1$ through $\ttrip_m$ ordered by largest vertex, with vertex $a\in V(T_i)$.

We define subsets
\begin{equation*}
    A=\{\forest\in \FT(G):\forest\vert_{G'} \text{ is not fixed under } \varphi'\},
\end{equation*}
and $B=\FT(G)\setminus A$. Now, say that $\forest\in A$. Then $\varphi(\forest):A\rightarrow A$ is defined as
\begin{equation*}
    \varphi(\forest)=(\forest\vert_{C_a}) + \varphi'(\forest \vert_{G'}).
\end{equation*}
Since $\varphi'$ is a first-preserving involution, it preserves $\alpha_1^{(i+1)}$, meaning that $\forest\vert_{C_a}+\varphi'(\forest\vert_{G'})$ is defined. Notice that
\begin{equation*}\varphi(\varphi(\forest))=(\forest\vert_{C_a})+\varphi'(\varphi'(\forest\vert_{G'}))=(\forest\vert_{C_a})+(\forest\vert_{G'})=\forest,\end{equation*} meaning $\varphi$ is an involution. It preserves $\type(\forest)$ and reverses $\sign(\forest)$, and thus is a sign-reversing involution. Additionally, it preserves $\alpha_1$ and $r$ for $\ttrip_1$ through $\ttrip_{i-1}$, and so it must be a first-preserving involution. Note $\varphi:A\rightarrow A$ has no fixed points.

The remaining forest triples $\forest\in B$ have $\forest\vert_{G'}$ fixed under $\varphi'$. We pick some $\forest'\in\FT(G')$ where $\forest'$ is fixed under $\varphi'$ and $r'_{\min}=1$. We look at forest triples where $\forest\vert_{G'}=\forest'$, and denote this subset $B_{\forest'}$.

Let $\first(\forest)=\{\ttrip_1,\ldots,\ttrip_i\}$ and $\last(\forest)=\{\ttrip_{i+1},\ldots,\ttrip_l\}$. Note that $\first(\forest)\in \FT'(C_a+T_i\vert_{G'})=\FT'(C_a+T'_{\min})$, where $\ttrip'_{\min}\in \forest'$. By Lemma~\ref{lem:mutatedtadpole}, $\FT'(C_a+T'_{\min})$ has a first-preserving involution, which we will call $\varphi_{\forest'}$. We can define a first-preserving involution $\varphi:B_{\forest'}\rightarrow B_{\forest'}$, where 
\begin{equation*}
    \varphi(\forest\in B_{\forest})=\varphi_{\forest'}(\first(\forest))\cup \last(\forest).
\end{equation*}
Then, we have
\begin{equation*}
    \varphi(\varphi(\forest))=\varphi_{\forest'}(\varphi_{\forest'}(\first(\forest)))\cup \last(\forest)=\first(\forest)\cup \last(\forest)=\forest,
\end{equation*}
so $\varphi$ is an involution. Since $\varphi_{\forest'}$ is a first-preserving involution, we see $\varphi$ preserves $\type(\forest)$ and $\alpha_1^{(\min)}$ and $r_{\min}$. Additionally, $\varphi$ reverses $\sign(\forest)$, unless $\forest$ is a fixed point. These fixed points in $B_{\forest'}$ under $\varphi$ have $\first(\forest)$ fixed under $\varphi_{\forest'}$, meaning fixed points are unit. Thus, $\varphi:B_{\forest'}\rightarrow B_{\forest'}$ is a first-preserving involution.

Note that for $\forest\in B_{\forest'}$, then 
\begin{equation*}
    \type(\last(\forest))=\type'(\forest'),\; e_{\type'(\forest)}=e_{\type'(\first(\forest))}\cdot e_{\type(\last(\forest))}, \; \sign(\forest)=\sign(\first(\forest)).
\end{equation*}

We can now combine the first-preserving involutions on $A$ and $B$, to make a new first-preserving involution on the set $\FT(G)$, and we can then calculate the value of $X_G^{(i)}(\boldsymbol{x})$ as
{\allowdisplaybreaks
\begin{align*}
    X^{(i)}_G(\boldsymbol{x})&=\sum_{\substack{\forest\in\FT^{(i)}(G)\\\text{$\forest$ fixed under $\varphi$}}}\sign(\forest)\cdot e_{\type'(\forest)}\\
    &=\sum_{\substack{\forest'\in\FT(G')\\r_{\min}=1\\\text{$\forest'$ fixed under $\varphi'$}}}\sum_{\forest\in (B_{\forest'})\cap(\FT^{(i)}(G))}\sign(\forest)\cdot e_{\type'(\forest)}\\
    &=\sum_{k=1}^{|G'|}\sum_{\substack{\forest'\in\FT^{(k)}(G')\\\text{$\forest'$ fixed under $\varphi'$}}}\sum_{\substack{\forest\in \FT^{(i)}(G)\\\forest\vert_{G'}=\forest'}}\sign(\forest)\cdot e_{\type'(\forest)}\\
    &=\sum_{k=1}^{|G'|}\sum_{\substack{\forest'\in\FT^{(k)}(G')\\\text{$\forest'$ fixed under $\varphi$}}}\Big(e_{\type'(\forest')}\cdot\sum_{\substack{\forest\in\FT^{(i)}(G)\\\forest\vert_{G'}=\forest'}}\sign(\first(\forest))\cdot e_{\type'(\first(\forest))}\Big)\\
    &=\sum_{k=1}^{|G'|}\sum_{\substack{\forest'\in\FT^{(k)}(G')\\\text{$\forest'$ fixed under $\varphi$}}}\Big(e_{\type'(\forest')}\cdot \sum_{\substack{\forest\in\FT'^{(i)}(C_a+T'_{\min})}}\sign(\forest)\cdot e_{\type'(\forest)}\Big)\\&=\sum_{k=1}^{|G'|}\sum_{\substack{\forest'\in\FT^{(k)}(G')\\\text{$\forest'$ fixed under $\varphi$}}}\Big(e_{\type'(\forest')}\cdot B_{a,k}^{(i)}(\boldsymbol{x})\Big)=\sum_{k=1}^{|G'|}X_{G'}^{(k)}(\boldsymbol{x})\cdot B_{a,k}^{(i)}(\boldsymbol{x}),
\end{align*}
}
which is Equation~\ref{eqn:csfC_a+G'spec}. Then, using this equation with Equation~\ref{eqn:nicecsftocsf} and Equation~\ref{eqn:nicebftobf}, we can easily derive Equation~\ref{eqn:csfC_a+G'}.
\end{proof}

We can now prove the $e$-positivity of adjacent cycle chains.
\begin{definition}
    An \emph{adjacent cycle chain} is a graph of the form $C_{a_1}+\cdots+C_{a_m}$, where $m\geq 1$ and $a_i\geq 2$ for $1\leq i\leq m$. The labeling of the vertices and ordering of the edges of the cycles are defined in Section~\ref{sec:cgi}.
\end{definition}
\begin{example}
Figure~\ref{fig:fakecyclechains} shows the adjacent cycle chain $C_2+C_4+C_3$ on the left, next to a graph that is not an adjacent cycle chain on the right. Note that the cut vertices of adjacent cycle chains are adjacent to each other, while the graph on the right has non-adjacent cut vertices.
\end{example}
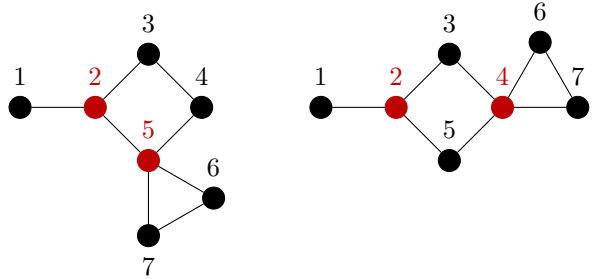
\begin{figure}[!htb]
\centering
\caption{Two graphs with red cut vertices, on the left an adjacent cycle chain, on the right a non-adjacent cycle chain.}
\label{fig:fakecyclechains}
\begin{tikzpicture}
\tikzmath{\o1=4;}
\vertex{1}[(\o1,0)]{(0,0)}['];
\vertex[red]{2}[(\o1,0)]{(1,0)}['];
\vertex{3}[(\o1,0)]{({1+sqrt(2)/2},{sqrt(2)/2})}['];
\vertex[red]{4}[(\o1,0)]{({1+sqrt(2)},{0})}['];
\vertex{5}[(\o1,0)]{({1+sqrt(2)/2},{-sqrt(2)/2})}['];
\vertex{7}[(\o1,0)]{({2+sqrt(2)},{0})}['];
\vertex{6}[(\o1,0)]{({1.5+sqrt(2)},{sqrt(3)/2})}['];
\draw (v'1)--(v'2)--(v'3)--(v'4)--(v'5)--(v'2) (v'4)--(v'6)--(v'7)--(v'4);
\vertex{1}{(0,0)};
\vertex[red]{2}{(1,0)};
\vertex{3}{({1+sqrt(2)/2},{sqrt(2)/2})};
\vertex{4}{({1+sqrt(2)},{0})};
\vertex[red]{5}{({1+sqrt(2)/2},{-sqrt(2)/2})};
\vertex[black][below]{7}{({1+sqrt(2)/2},{-1-sqrt(2)/2})};
\vertex{6}{({1+sqrt(3)/2+sqrt(2)/2},{-0.5-sqrt(2)/2})};
\draw (v1)--(v2)--(v3)--(v4)--(v5)--(v2) (v5)--(v6)--(v7)--(v5);
\end{tikzpicture}
\end{figure}

\begin{corollary}
    Adjacent cycle chains are $e$-positive.
\end{corollary}
\begin{proof}
We use induction: if the adjacent cycle chain is just $C_a$, then it has a first-preserving involution by Theorem~\ref{thm:cg}. Now, assume that all adjacent cycle chains with $m$ cycles have a first-preserving involution. Then an adjacent cycle chain with $m+1$ cycles can be written as $C_a+G'$, where $G'$ is an adjacent cycle chain with $m$ cycles and thus has a first-preserving involution. By Theorem~\ref{thm:c+g}, $C_a+G'$ also has a first-preserving involution, which means by induction all finite adjacent cycle chains have first-preserving involutions and are thus $e$-positive.
\end{proof}
This result provides an alternate proof for the $e$-positivity of tadpoles ($P_a+C_b$), hats ($P_a+C_b+P_c$), and dumbbells ($C_a+P_b+C_c$), which have been proven by previous papers \cite{mitrovic2024epositivitynewclasses,wang2021cyclechordgraphsepositive,wang2024compositionmethodchromaticsymmetric}. We can extend this idea to adjacent cycle+clique chains.
\begin{definition}
    An \emph{adjacent cycle+clique chain} is a graph of the form $G_1+\cdots + G_m$, where $m\geq 1$ and $G_i$ is either a clique or a cycle.
\end{definition}
\begin{corollary}
    Adjacent cycle+clique chains are $e$-positive.
\end{corollary}
\begin{proof}
    If a graph $G'$ has a first-preserving involution, then the graph $K_a+G'$ also has a first-preserving involution \cite[Theorem~4.10]{tom2024signed}. For the base case, both $C_a$ and $K_a$ have a first-preserving involution. Now, assume all adjacent cycle+clique chains with $m$ cycles and cliques have a first-preserving involution. Then, an adjacent cycle+clique chain with $m+1$ cycles, $G$, can be written as either $C_a+G'$ or $K_a+G'$, where $G'$ is an adjacent cycle+clique chain with $m$ cycles. Thus, $G$ also has a first-preserving involution, and by induction, all finite adjacent cycle+clique chains have a first-preserving involution.
\end{proof}
\begin{corollary}
    The chromatic symmetric function of $C_a+C_b$ is
    \begin{multline}
        \label{eqn:C_a+C_b}
    X_{C_a+C_b}(\boldsymbol{x})=\sum_{\substack{\alpha\models a,\beta\models b+\alpha_1\\\len(\beta)\geq 2,\beta_1\leq \alpha_1,\beta_2\geq \alpha_1}}(\alpha_1-1)\cdots (\alpha_l-1)\cdot(\beta_2-\beta_1+1)(\beta_1+\beta_2-\alpha_1-1)\cdot\\[-20pt](\beta_3-1)\cdots (\beta_l-1)\cdot e_{\sort(\alpha\setminus \alpha_1\cdot \beta\setminus\beta_1 \cdot (\beta_1-1))}.
    \end{multline}
\end{corollary}
\begin{proof}
    From Equation~\ref{eqn:cyclecsfspec} and Equation~\ref{eqn:csfC_a+G'}, we get that
    \begin{equation*}X_{C_a+C_b}(\boldsymbol{x})=\sum_{k=1}^{a}\left(\sum_{\alpha\models a-k}e_{\sort(\alpha)}\cdot (k-1)\cdot (\alpha_1-1)\cdots (\alpha_l-1)\cdot B_{b,k}(\boldsymbol{x})\right).
    \end{equation*}
    We can combine the sum over $k$ and the sum over $\alpha\models a-k$ to become a sum over $\alpha\models a$, where $\alpha_1$ now represents the value of $k$. This gives
    \begin{multline*}
        \sum_{\alpha\models a}e_{\sort(\alpha\setminus \alpha_1)}\cdot (\alpha_1-1)\cdots (\alpha_l-1)\cdot\\[-2mm]\sum_{\substack{\beta\models b+\alpha_1, \len(\beta)\geq 2\\\beta_1\leq \alpha_1,\beta_2\geq \alpha_1}}(\beta_2-\beta_1+1)(\beta_1+\beta_2-\alpha_1-1)\cdot (\beta_3-1)\cdots (\beta_l-1) \cdot e_{\sort(\beta_1-1,\beta_2,\ldots,\beta_l)},
    \end{multline*}
    which simplifies to Equation~\ref{eqn:C_a+C_b}.
\end{proof}
\begin{example}
    The chromatic symmetric function of $C_4+C_3$ is
    \begin{equation*}
        X_{C_4+C_3}\oset[0.31cm]{\alpha,\beta\hspace{0.1cm}=}{(\boldsymbol{x})}=\overbrace{4e_{3,2,1}}^{(2,2),(2,3)}+\overbrace{8e_{4,2}}^{(2,2),(1,4)}+\overbrace{12e_{4,2}}^{(4),(3,4)}+\overbrace{24e_{5,1}}^{(4),(2,5)}+\overbrace{36e_{6}}^{(4),(1,6)}.
    \end{equation*}
\end{example}
\section{Cycle Plus Tree Involution}\label{sec:cpti}
In this section we will prove Lemma~\ref{lem:mutatedtadpole}, which states the subset $\FT'(C_a+U_k)$ has a first-preserving involution and provides an explicit formula for $B_{a,k}(\boldsymbol{x})$. As a recap, we will redefine these terms below.

\begin{definition}
The graph $C_a+U_k$ is a cycle with $a$ vertices connected to a tree with $k$ vertices. For every $\forest\in\FT(C_a+U_k)$, let $\ttrip=(T,\alpha,r)$ where $a\in V(T)$. The set $\FT'(C_a+U_k)$ is the set of forest triples where $E(U_k)\subseteq E(T)$ and $\alpha_l\geq k$.

Then, $\FT'^{(i)}(C_a+U_k)$ is the subset of $\FT'(C_a+U_k)$ where $\alpha_1^{(\min)}=i$ and $r_{\min}=1$.
\end{definition}

We will define a way to denote forest triples similar to the manner described in Section~\ref{sec:cgi}. Say forest triple $\forest\in\FT'(C_a+U_k)$ has edge $(a,1)$. Then, we write $\forest=\{\ttrip_1,\ldots,\ttrip_m, \ttrip_{\min}\}$ as a forest triple with $m+1$ tree triples where $m\geq 0$. We order the tree triples so vertex $1\in V(T_{\min})$ and $\min(V(T_1))<\cdots<\min(V(T_m))$.

If instead edge $(a,1)$ is not in the forest triple, then we write $\forest=\{\ttrip_1,\ldots,\ttrip_m,\ttrip'=(T',\alpha',r'),\ttrip_{\min}\}$ as a forest triple with $m+2$ tree triples where $m\geq 0$. We order the tree triples so vertex $1\in V(T_{\min})$ and $\min(V(T_1))<\cdots<\min(V(T_m))<\min(V(T'))$.

Similar to the cycles, we can identify the tree of a tree triple based on a single vertex. If edge $(a,1)\in E(T_{\min})$, we write
\begin{equation*}
    \forest=\left\langle(v_1,\alpha^{(1)},r_1),\ldots,(v_{m},\alpha^{(m)},r_{m}),(v_{\min},\alpha^{(\min)},r_{\min})\right\rangle,
\end{equation*}
where $2\leq v_1<\cdots <v_m<v_{\min}\leq a$ and
\begin{gather*}
    E(T_i)=\{(v_i,v_i+1),\ldots,(v_i+|T_i|-2,v_i+|T_i|-1)\},\\E(T_{\min})=\{(v_{\min},v_{\min}+1),\ldots,(a-1,a),(a,1),(1,2),\ldots, (v_1-2,v_1-1)\}\cup E(U_k).
\end{gather*}
If instead edge $(a,1)\not\in E(T_{\min})$, we write
\begin{equation*}
    \forest=\left\langle(v_1,\alpha^{(1)},r_1),\ldots,(v_{m},\alpha^{(m)},r_{m}),(v',\alpha',r'),(v_{\min},\alpha^{(\min)},r)\right\rangle,
\end{equation*}
where $1=v_{\min}<v_1<\cdots<v_m<v'\leq a$ and
\begin{gather*}
    E(T_i)=\{(v_i,v_i+1),\ldots,(v_i+|T_i|-2,v_i+|T_i|-1)\},\\E(T')=\{(v',v'+1),\ldots,(a-1,a)\}\cup E(U_k),\\
    E(T_{\min})=\{(1,2),\ldots, (|T_{\min}|-1,|T_{\min}|)\}.
\end{gather*}
Note that if $(a,1)\in E(T_{\min})$, then for $\forest\in\FT'(C_a+U_k)$ we need $\alpha_l^{(\min)}\geq k$. If $(a,1)\not\in E(T_{\min})$, then for $\forest\in \FT'(C_a+U_k)$, we need $\alpha'_l\geq k$.

\begin{examplen}
    Figure~\ref{fig:CaTk-breakjoin} and Figure~\ref{fig:CaTk-shiftbreakjoin} provide examples of forest triples.
\end{examplen}
\begin{proof}[Proof of Lemma~\ref{lem:mutatedtadpole}]

A first-preserving involution always preserves $\alpha_1^{(\min)}$ and $r_{\min}$ for $\forest\in\FT'(C_a+U_k)$. We will define this involution in 5 steps.
\paragraph{Step 1:}
This step is nearly identical to Step 1 in the cycle involution. We define sets $A$ and $B$ with indexed subsets $A_i$ and $B_i$ for $i\in \mathbb{N}$ as
\begin{gather*}
    A_i=\left\{\forest:m \geq i\scolon \len(\alpha^{(1)})=\cdots=\len(\alpha^{(i)})=1\scolon r_1,\ldots,r_{i-1}\geq 2\scolon r_i=1\right\},\\
    B_i=\left\{\forest:m\geq i\scolon\len(\alpha^{(i)})\geq 2\scolon\len(\alpha^{(1)})=\cdots=\len(\alpha^{(i-1)})=1\scolon r_1,\ldots,r_{i-1}\geq 2\right\}.
\end{gather*}
Intuitively, if $\forest\in A_i$ or $\forest\in B_i$, then $i$ is the smallest integer such that either $\len(\alpha^{(i)})\geq 2$ or $r_i=1$. We define function $\tjoin_i:A_i\rightarrow B_i$, where $\tjoin_i(\forest)$ replaces $\ttrip_i$ and $\ttrip_{i+1}$ in $\forest$ with tree triples
\begin{equation*}
   S=\left(v_{i},\alpha^{(i+1)}\cdot\alpha^{(i)},r_{i+1}\right).
\end{equation*}
The inverse map $\tbreak_i:B_i\rightarrow A_i$ replaces $\ttrip_i$ in $\forest$ with
\begin{equation*}
    S_1=\left(v_i,(\alpha_l^{(i)}),1\right),S_2=\left(v_i+\alpha_l^{(i)},\alpha^{(i)}\setminus \alpha_l^{(i)}, r_i\right).
\end{equation*}

If $\forest\in A_i$, then tree triples $\ttrip_1$ through $\ttrip_{i-1}$ all have $\len(\alpha)=1$ and $r\geq 2$. Those tree triples are untouched in $\tjoin_i(\forest)$, while $\ttrip_i$ and $\ttrip_{i+1}$ are replaced by $S$. Since $S$ has a composition with length at least 2, then $\tjoin_i(\forest)\in B_i$. Applying $\tbreak_i(\tjoin_i(\forest))$ breaks $S$ while preserving the other tree triples, resulting in $\forest$ again. Thus, $\tjoin_i$ and $\tbreak_i$ are inverses, meaning $\tjoin_i$ is a bijection.
\begin{examplen}
    Figure~\ref{fig:CaTk-breakjoin} shows an example, where we either join the light blue and dark blue tree triples or break the blue tree triple.
\end{examplen}
\begin{figure}[!htb]
    \centering
    \caption{Example of $\tjoin$ and $\tbreak$ for $\forest\in \FT'(C_6+P_4)$}
    \label{fig:CaTk-breakjoin}
\begin{tikzpicture}
\tikzmath{\h0=(2.75-(1+2*cos(60))+3.5+3;}
\vertex[red]{1}[(\h0,0)]{({cos(60)},{sin(60)})}
\vertex[blue]{2}[(\h0,0)]{({1+cos(60)},{sin(60)})}
\vertex[blue]{3}[(\h0,0)]{({1+2*cos(60)},0)}
\vertex[blue]{4}[(\h0,0)]{({1+cos(60)},{-sin(60)})}
\vertex[blue]{5}[(\h0,0)]{({cos(60)},{-sin(60)})}
\vertex[red]{6}[(\h0,0)]{(0,0)}
\vertex[red]{7}[(\h0,0)]{(-1,0)}
\vertex[red]{8}[(\h0,0)]{(-2,0)}
\vertex[red]{9}[(\h0,0)]{(-3,0)}
\draw[red] (v1)--(v6)--(v7)--(v8)--(v9);
\draw[blue] (v2)--(v3)--(v4)--(v5);
\node[above right=0.2cm and 0.1cm of v3] {$\textcolor{blue}{S=\ttrip_1}$};
\node[below left=0.2cm and 0.5cm of v6] {$\textcolor{red}{\ttrip_{\min}}$};
\node[align=center,text width=5cm] at ({(-3+2*cos(60)+1)/2+\h0},2) {${\forest_B\in B_1,\quad\forest_B=\tjoin_1(\forest_A)=}$\\$\left\langle\textcolor{blue}{\Big(2,(1,2,\textcolor{lightblue}{\textbf{1}}),1\Big)},\textcolor{red}{\Big(6,(1,4),1\Big)}\right\rangle$};
\draw[<->, line width=0.1cm] (2.75,0) -- (3.5,0);
\tikzmath{\h1=0;}
\vertex[red]{1}[(\h1,0)]{({cos(60)},{sin(60)})}[']
\vertex[lightblue]{2}[(\h1,0)]{({1+cos(60)},{sin(60)})}[']
\vertex[blue]{3}[(\h1,0)]{({1+2*cos(60)},0)}[']
\vertex[blue]{4}[(\h1,0)]{({1+cos(60)},{-sin(60)})}[']
\vertex[blue]{5}[(\h1,0)]{({cos(60)},{-sin(60)})}[']
\vertex[red]{6}[(\h1,0)]{(0,0)}[']
\vertex[red]{7}[(\h1,0)]{(-1,0)}[']
\vertex[red]{8}[(\h1,0)]{(-2,0)}[']
\vertex[red]{9}[(\h1,0)]{(-3,0)}[']
\draw[red] (v'1)--(v'6)--(v'7)--(v'8)--(v'9);
\draw[blue] (v'3)--(v'4)--(v'5);
\node[text width=6cm, align=center] at ({(-3+2*cos(60)+1)/2+\h1-0.5},2) {${\forest_A\in A_1,\quad\forest_A=\tbreak_1(\forest_B)=}$\\$\left\langle\textcolor{lightblue}{\Big(2,(\textbf{1}),1\Big)},\textcolor{blue}{\Big(3,(1,2),1\Big)},\textcolor{red}{\Big(6,(1,4),1\Big)}\right\rangle$};
\node[above right=-0.3cm and 0.1cm of v'2] {$\textcolor{lightblue}{S_1=\ttrip_1}$};
\node[below right=0.2cm and 0.0cm of v'3] {$\textcolor{blue}{S_2=\ttrip_2}$};
\node[below left=0.2cm and 0.5cm of v'6] {$\textcolor{red}{\ttrip_{\min}}$};
\end{tikzpicture}
\end{figure}
\paragraph{Step 2:}
This time, we focus on forest triples where edge $(a,1)\in E(T_{\min})$, defining $C$ and $D$ as
\begin{gather*}
    C=\{\forest:
    (a,1)\in E(T_{\min})\scolon\len(\alpha^{(1)})=\cdots=\len(\alpha^{(m)})=1\scolon r_1,\ldots,r_{m-1}\geq 2\scolon r_m=1\scolon\\
    \text{either } \len(\alpha^{(\min)})\geq 2 \text{ or } |T_m|\geq k\},\\
    D=\{\forest: (a,1)\in E(T_{\min})\scolon\len(\alpha^{(1)})=\cdots=\len(\alpha^{(m)})=1\scolon r_1,\ldots,r_m\geq 2\scolon\\\len(\alpha^{(\min)})\geq 2\scolon\alpha_2^{(\min)}\leq a-v_{\min}\}.
\end{gather*}
The bijection from $C$ to $D$ will be similar to $\tbreak_i$ and $\tjoin_i$, but instead of attaching the single part of $\alpha^{(m)}$ at the end of $\alpha^{(\min)}$, we insert it in the second position. We define $\secondjoin:C\rightarrow D$, where $\secondjoin(\forest)$ replaces $\ttrip_{m}$ and $\ttrip_{\min}$ with
\begin{equation*}
    S=\left(v_m,(\alpha_1^{(\min)})\cdot \alpha^{(m)}\cdot (\alpha^{(\min)}\setminus \alpha_1^{(\min)}), r_m\right),
\end{equation*}
with inverse map $\secondbreak:D\rightarrow C$ that replaces $\ttrip_{\min}$ in $\forest$ with tree triples
\begin{equation*}
   S_1=\left(v_{\min},\alpha_2^{(\min)},1\right), S_2=\left(v_{\min}+\alpha_2^{(\min)},\alpha^{(\min)}\setminus\alpha_2^{(\min)},r_{\min}\right).
\end{equation*}
If $\forest$ has $\len(\alpha^{(\min)})=1$, then it must have $|T_m|\geq k$ in order for $\secondjoin(\forest)$ to have $\alpha^{(\min)}_l\geq k$. The last condition is necessary for $\secondjoin(\forest)\in\FT'(C_a+U_k)$. This explains one of the restrictions for $C$.

All $\forest\in C$ have $\alpha_1^{(m)}=|T_m|\leq a-v_m$. Inspecting where $\alpha_1^{(m)}$ ends up after joining, we see that in $\secondjoin(\forest)$, both $\len(\alpha^{(\min)})\geq 2$ and $\alpha_2^{(\min)}\leq a-v_{\min}$, which explains those restrictions in set $D$. 
\begin{examplen}
    Figure~\ref{fig:CaTk-secondbreakjoin} shows an example of $\secondjoin$ and $\secondbreak$. Note that if $\forest_C$ had $\alpha^{(\min)}=(6)$, then we would not be able to secondjoin.
\end{examplen}
\begin{figure}[!htb]
\centering
    \caption{Example of $\secondjoin$ and $\secondbreak$ for $\forest\in \FT'(C_6+P_4)$.}
    \label{fig:CaTk-secondbreakjoin}
\begin{tikzpicture}
\vertex[blue]{1}{({cos(60)},{sin(60)})}
\vertex[red]{2}{({1+cos(60)},{sin(60)})}
\vertex[red]{3}{({1+2*cos(60)},0)}
\vertex[lightblue]{4}{({1+cos(60)},{-sin(60)})}
\vertex[blue]{5}{({cos(60)},{-sin(60)})}
\vertex[blue]{6}{(0,0)}
\vertex[blue]{7}{(-1,0)}
\vertex[blue]{8}{(-2,0)}
\vertex[blue]{9}{(-3,0)}
\draw[red] (v2)--(v3);
\draw[blue] (v1)--(v6)--(v7)--(v8)--(v9) (v5)--(v6);
\node[above right=0.2cm and 0.1cm of v3] {$\textcolor{red}{\ttrip_1}$};
\node[below left=0.2cm and 0.5cm of v6] {$\textcolor{blue}{\ttrip_{\min}}$};
\node[below right=-0.3cm and 0.1cm of v4] {$\textcolor{lightblue}{\ttrip_{2}}$};
\node[align=center,text width=5cm] at ({(-3+2*cos(60)+1)/2},2) {${\forest_C\in C,\quad\forest_C=\secondbreak(\forest_D)=}$\\$\left\langle\textcolor{red}{\Big(2,(2),2\Big)},\textcolor{lightblue}{\Big(4,(\textbf{1}),1\Big)},\textcolor{blue}{\Big(5,(2,4),2\Big)}\right\rangle$};
\draw[<->, line width=0.1cm] (2.75,0) -- (3.5,0);
\tikzmath{\h1=(2.75-(1+2*cos(60))+3.5+3;}
\vertex[blue]{1}[(\h1,0)]{({cos(60)},{sin(60)})}[']
\vertex[red]{2}[(\h1,0)]{({1+cos(60)},{sin(60)})}[']
\vertex[red]{3}[(\h1,0)]{({1+2*cos(60)},0)}[']
\vertex[blue]{4}[(\h1,0)]{({1+cos(60)},{-sin(60)})}[']
\vertex[blue]{5}[(\h1,0)]{({cos(60)},{-sin(60)})}[']
\vertex[blue]{6}[(\h1,0)]{(0,0)}[']
\vertex[blue]{7}[(\h1,0)]{(-1,0)}[']
\vertex[blue]{8}[(\h1,0)]{(-2,0)}[']
\vertex[blue]{9}[(\h1,0)]{(-3,0)}[']
\draw[blue] (v'1)--(v'6)--(v'7)--(v'8)--(v'9) (v'4)--(v'5)--(v'6);
\draw[red] (v'2)--(v'3);
\node[text width=5cm, align=center] at ({(-3+2*cos(60)+1)/2+\h1-0.5},2) {${\forest_D\in D,\quad\forest_D=\secondjoin(\forest_C)=}$\\$\left\langle\textcolor{red}{\Big(2,(2),2\Big)},\textcolor{blue}{\Big(4,(2,\textcolor{lightblue}{\textbf{1}},4),2\Big)}\right\rangle$};
\node[above right=0.2cm and 0.1cm of v'3] {$\textcolor{red}{\ttrip_1}$};
\node[below left=0.2cm and 0.5cm of v'6] {$\textcolor{blue}{\ttrip_{\min}}$};
\end{tikzpicture}
\end{figure}
\paragraph{Step 3:}
We now focus on forest triples where edge $(a,1)\not\in E(T_{\min})$, with
\begin{gather*}
    E=\{\forest:(a,1)\not\in E(T_{\min})\scolon\len(\alpha^{(1)})=\cdots=\len(\alpha^{(m)})=1\scolon r_1,\ldots,r_{m-1}\geq 2\scolon r_m=1
    \},\\
    F=\{\forest:(a,1)\not\in E(T_{\min})\scolon\len(\alpha^{(1)})=\cdots=\len(\alpha^{(m)})=1\scolon r_1,\ldots,r_{m}\geq 2\scolon\len(\alpha^{(\min)})\geq 2\}.
\end{gather*}
We define bijection $\shiftjoin:E\rightarrow F$ where
\begin{multline*}
    \shiftjoin(\forest)=\Big\langle\!\left(v_1+\alpha_1^{(m)},\alpha^{(1)},r_1\right),\ldots,\left(v_{m-1}+\alpha_1^{(m)},\alpha^{(m-1)},r_{m-1}\right),\\
    \ttrip',\left(1, \alpha^{(\min)}\cdot \alpha^{(m)},r_{\min}\right)\!\Big\rangle,
\end{multline*}
with inverse $\shiftbreak:F\rightarrow E$ with
\begin{multline*}
    \shiftbreak(\forest)=\Big\langle\!\left(v_1-\alpha_l^{(\min)},\alpha^{(1)},r_1\right),\ldots,\left(v_{m}-\alpha_l^{(\min)},\alpha^{(m)},r_m\right),\\
    \left(v'-\alpha_l^{(\min)},(\alpha_l^{(\min)}),1\right),\ttrip',\left(1,\alpha^{(\min)}\setminus \alpha_l^{(\min)},r_{\min}\right)\!\Big\rangle.
\end{multline*}
Essentially, we are either joining $\ttrip_m$ to $\ttrip_{\min}$ or breaking $\ttrip_{\min}$, while ignoring $\ttrip'$.
\begin{examplen}
    Figure~\ref{fig:CaTk-shiftbreakjoin} shows an example of $\shiftjoin$ and $\shiftbreak$. Note that $\ttrip'$ (which is the red tree triple in the diagram) remains untouched.    
\end{examplen}
\begin{figure}[!htb]
\centering
    \caption{Example of $\shiftjoin$ and $\shiftbreak$ for $\forest\in\FT'(C_6+P_4)$.}
    \label{fig:CaTk-shiftbreakjoin}
\begin{tikzpicture}
\vertex[blue]{1}{({cos(60)},{sin(60)})}
\vertex[darkgreen]{2}{({1+cos(60)},{sin(60)})}
\vertex[darkgreen]{3}{({1+2*cos(60)},0)}
\vertex[lightblue]{4}{({1+cos(60)},{-sin(60)})}
\vertex[lightblue]{5}{({cos(60)},{-sin(60)})}
\vertex[red]{6}{(0,0)}
\vertex[red]{7}{(-1,0)}
\vertex[red]{8}{(-2,0)}
\vertex[red]{9}{(-3,0)}
\draw[red] (v6)--(v7)--(v8)--(v9);
\draw[darkgreen] (v2)--(v3);
\draw[lightblue] (v4)--(v5);
\node[align=center,text width=5cm] at ({(-3+2*cos(60)+1)/2},2) {${\forest_E\in E,\quad\forest_E=\shiftbreak(\forest_F)=}$\\$\left\langle\textcolor{darkgreen}{\Big(2,(2),2\Big)},\textcolor{lightblue}{\Big(4,(\textbf{2}),1\Big)},\textcolor{red}{\Big(6,(4),2\Big)},\textcolor{blue}{\Big(1,(1),1\Big)}\right\rangle$};
\node[above right=0.2cm and 0.1cm of v3] {$\textcolor{darkgreen}{\ttrip_1}$};
\node[below left=0.2cm and 0.5cm of v6] {$\textcolor{red}{\ttrip'}$};
\node[below right=-0.3cm and 0.1cm of v4] {$\textcolor{lightblue}{\ttrip_2}$};
\node[above left=-0.3cm and 0.2cm of v1] {$\textcolor{blue}{\ttrip_{\min}}$};
\draw[<->, line width=0.1cm] (3.25,0) -- (4,0);
\tikzmath{\h1=(2.75-(1+2*cos(60))+4.5+3;}
\vertex[blue]{1}[(\h1,0)]{({cos(60)},{sin(60)})}[']
\vertex[blue]{2}[(\h1,0)]{({1+cos(60)},{sin(60)})}[']
\vertex[blue]{3}[(\h1,0)]{({1+2*cos(60)},0)}[']
\vertex[darkgreen]{4}[(\h1,0)]{({1+cos(60)},{-sin(60)})}[']
\vertex[darkgreen]{5}[(\h1,0)]{({cos(60)},{-sin(60)})}[']
\vertex[red]{6}[(\h1,0)]{(0,0)}[']
\vertex[red]{7}[(\h1,0)]{(-1,0)}[']
\vertex[red]{8}[(\h1,0)]{(-2,0)}[']
\vertex[red]{9}[(\h1,0)]{(-3,0)}[']
\draw[red] (v'6)--(v'7)--(v'8)--(v'9);
\draw[darkgreen] (v'4)--(v'5);
\draw[blue] (v'1)--(v'2)--(v'3);
\node[text width=5cm, align=center] at ({(-3+2*cos(60)+1)/2+\h1-0.5},2) {${\forest_F\in F,\quad\forest_F=\shiftjoin(\forest_E)=}$\\$\left\langle\textcolor{darkgreen}{\Big(4,(2),2\Big)},\textcolor{red}{\Big(6,(4),2\Big)},\textcolor{blue}{\Big(1,(1,\textcolor{lightblue}{\textbf{2}}),1\Big)}\right\rangle$};
\node[above right=0.2cm and 0.1cm of v'3] {$\textcolor{blue}{\ttrip_{\min}}$};
\node[below left=0.2cm and 0.5cm of v'6] {$\textcolor{red}{\ttrip'}$};
\node[below right=-0.3cm and 0.1cm of v'4] {$\textcolor{darkgreen}{\ttrip_2}$};
\end{tikzpicture}
\end{figure}
\paragraph{Step 4:}
We define sets $G$ and $H$ with indexed subsets $G_i, H_i$ for $i\in\mathbb{Z}$ where
\begin{gather*}
    G_i=\{\forest:(a,1)\not\in E(T_{\min})\scolon\len(\alpha^{(1)})=\cdots=\len(\alpha^{(m)})=1\scolon r_1,\ldots,r_m\geq 2\scolon \\\len(\alpha^{(\min)})=1\scolon r'\leq |T'|+|T_{\min}|-k\scolon |T'|-k-r'+1=i\},\\
    H_i=\{\forest:(a,1)\in E(T_{\min})\scolon \len(\alpha^{(1)})=\cdots=\len(\alpha^{(m)})=1\scolon r_1,\ldots,r_{m}\geq 2\scolon \\\len(\alpha^{(\min)})\geq 2\scolon\alpha_2^{(\min)}\geq a-v_{\min}+1\scolon v_1-1-\alpha_1^{(\min)}=i\},
\end{gather*}
where for the final restriction on $H_i$, if $\forest=\{\ttrip_{\min}\}$ has $m=0$, then $v_1=v_{\min}$.
We define a bijection, $\rotatejoin_i:G_i\rightarrow H_i$, with
\begin{multline*}
    \rotatejoin_i(\forest)=\Big\langle\!\left(v_1+i,\alpha^{(1)},r_1\right),\ldots,\left(v_m+i,\alpha^{(m)},r_m\right),\\\left(a-r'+1,\alpha^{(\min)}\cdot \alpha',r_{\min}\right)\!\Big\rangle.
\end{multline*}
Say $\forest_G\in G_i$ and $\forest_H=\rotatejoin_i(\forest_G)$. Note that in $\forest_G$, we have $v_1=|T_{\min}|+1=\alpha_1^{(\min)}+1$. After rotating, $v_1$ in $\forest_G$ becomes $v_1+|T'|-k-r'+1$ in $\forest_H$. In other words,
\begin{align*}
    \text{$(v_1)$ in $\forest_H$}&=\text{$(v_1+|T'|-k-r'+1)$ in $\forest_G$, so}\\
    \text{$(v_1-1-\alpha_1^{(\min)})$ in $\forest_H$}&=\text{$(v_1+|T'|-k-r'+1-1-\alpha_1^{(\min)})$ in $\forest_G$}\\
    &=|T'|-k-r'+1=i,
\end{align*}
meaning $\forest_H\in H_i$. The inverse $\rotatebreak_i:H_i\rightarrow G_i$ is defined as
\begin{multline*}
    \rotatebreak_i(\forest)=\Big\langle\!\left(v_1-i,\alpha^{(1)},r_1\right),\ldots,\left(v_m-i,\alpha^{(m)},r_m\right),\\
    \left(v_{\min}-i,\alpha^{(\min)}\setminus \alpha_1^{(\min)},a-v_{\min}+1\right),\left(1,(\alpha_1^{(\min)}),r_{\min}\right)\!\Big\rangle.
\end{multline*}
\begin{examplen}
    Figure~\ref{fig:CaTk-rotatebj} shows examples of $\rotatejoin_i$ and $\rotatebreak_i$. Subfigure~\ref{subfig:CaTk-rotatebj-norm} and Subfigure~\ref{subfig:CaTk-rotatebj-norm2} show how the value of $r'$ in $\forest_G$ determines the ``rotation" of $\rotatejoin_i(\forest_G)$.

    Subfigure~\ref{subfig:CaTk-rotatebj-3} has $\forest_G$ where $r'=5=|T'|+|T_{\min}|-k$. If $r'=6$, then $\rotatejoin_i(\forest_G)$ would break edge $(1,6)$, which would either mess the order of tree triples or in this example, result in a forest that is not a non-broken circuit. If $r'=7$, then $\rotatejoin_i(\forest_G)$ would be the same graph as when $r'=1$. Thus, if $r'\geq |T'|+|T_{\min}|-k+1$, we cannot $\rotatejoin$.
\end{examplen}
\begin{figure}[!htb]
\caption{Examples of $\rotatejoin_i$ and $\rotatebreak_i$ for forest triples in $\FT'(C_6+P_4)$.}
    \label{fig:CaTk-rotatebj}
\begin{subfigure}{\textwidth}
\centering
    \caption{Simple example of rotate-joining where $r'=1$ in $F_G$.}
    \label{subfig:CaTk-rotatebj-norm}
\begin{tikzpicture}
\vertex[blue]{1}{({cos(60)},{sin(60)})}
\vertex[red]{2}{({1+cos(60)},{sin(60)})}
\vertex[red]{3}{({1+2*cos(60)},0)}
\vertex[lightblue]{4}{({1+cos(60)},{-sin(60)})}
\vertex[lightblue]{5}{({cos(60)},{-sin(60)})}
\vertex[lightblue]{6}{(0,0)}
\vertex[lightblue]{7}{(-1,0)}
\vertex[lightblue]{8}{(-2,0)}
\vertex[lightblue]{9}{(-3,0)}
\draw[lightblue] (v4)--(v5)--(v6)--(v7)--(v8)--(v9);
\draw[red] (v2)--(v3);
\node[above right=0.2cm and 0.1cm of v3] {$\textcolor{red}{\ttrip_1}$};
\node[below left=0.2cm and 0.5cm of v6] {$\textcolor{lightblue}{\ttrip'}$};
\node[above left=-0.3cm and 0.2cm of v1] {$\textcolor{blue}{\ttrip_{\min}}$};
\node[align=center,text width=5.5cm] at ({(-3+2*cos(60)+1)/2},2) {${\forest_G\in G_2,\quad\forest_G=\rotatebreak_2(\forest_H)=}$\\$\left\langle\textcolor{red}{\Big(2,(2),2\Big)},\textcolor{lightblue}{\Big(4,(\textbf{2,4}),1\Big)},\textcolor{blue}{\Big(1,(1),1\Big)}\right\rangle$};
\draw[<->, line width=0.1cm] (3.25,0) -- (4,0);
\tikzmath{\h1=(2.75-(1+2*cos(60))+4.5+3;}
\vertex[blue]{1}[(\h1,0)]{({cos(60)},{sin(60)})}[']
\vertex[blue]{2}[(\h1,0)]{({1+cos(60)},{sin(60)})}[']
\vertex[blue]{3}[(\h1,0)]{({1+2*cos(60)},0)}[']
\vertex[red]{4}[(\h1,0)]{({1+cos(60)},{-sin(60)})}[']
\vertex[red]{5}[(\h1,0)]{({cos(60)},{-sin(60)})}[']
\vertex[blue]{6}[(\h1,0)]{(0,0)}[']
\vertex[blue]{7}[(\h1,0)]{(-1,0)}[']
\vertex[blue]{8}[(\h1,0)]{(-2,0)}[']
\vertex[blue]{9}[(\h1,0)]{(-3,0)}[']
\draw[blue] (v'3)--(v'2)--(v'1)--(v'6)--(v'7)--(v'8)--(v'9);
\draw[red] (v'4)--(v'5);
\node[above left=-0.3cm and 0.1cm of v'5] {$\textcolor{red}{\ttrip_1}$};
\node[above left=-0.3cm and 0.2cm of v'1] {$\textcolor{blue}{\ttrip_{\min}}$};
\node[text width=5.5cm, align=center] at ({(-3+2*cos(60)+1)/2+\h1-0.5},2) {${\forest_H\in H_2,\quad\forest_H=\rotatejoin_2(\forest_G)=}$\\$\left\langle\textcolor{red}{\Big(4,(2),2\Big)},\textcolor{blue}{\Big(6,(1,\textcolor{lightblue}{\textbf{2,4}}),1\Big)}\right\rangle$};
\end{tikzpicture}
\end{subfigure}\vspace{2mm}
\begin{subfigure}{\textwidth}
\centering
   \caption{Same example as previous, but where $r'=2$ in $F_G$.}
    \label{subfig:CaTk-rotatebj-norm2}
\begin{tikzpicture}
\vertex[blue]{1}{({cos(60)},{sin(60)})}
\vertex[red]{2}{({1+cos(60)},{sin(60)})}
\vertex[red]{3}{({1+2*cos(60)},0)}
\vertex[lightblue]{4}{({1+cos(60)},{-sin(60)})}
\vertex[lightblue]{5}{({cos(60)},{-sin(60)})}
\vertex[lightblue]{6}{(0,0)}
\vertex[lightblue]{7}{(-1,0)}
\vertex[lightblue]{8}{(-2,0)}
\vertex[lightblue]{9}{(-3,0)}
\draw[lightblue] (v4)--(v5)--(v6)--(v7)--(v8)--(v9);
\draw[red] (v2)--(v3);
\node[above right=0.2cm and 0.1cm of v3] {$\textcolor{red}{\ttrip_1}$};
\node[below left=0.2cm and 0.5cm of v6] {$\textcolor{lightblue}{\ttrip'}$};
\node[above left=-0.3cm and 0.2cm of v1] {$\textcolor{blue}{\ttrip_{\min}}$};
\node[align=center,text width=5.5cm] at ({(-3+2*cos(60)+1)/2},2) {${\forest_G\in G_1,\quad\forest_G=\rotatebreak_1(\forest_H)=}$\\$\left\langle\textcolor{red}{\Big(2,(2),2\Big)},\textcolor{lightblue}{\Big(4,(\textbf{2,4}),2\Big)},\textcolor{blue}{\Big(1,(1),1\Big)}\right\rangle$};
\draw[<->, line width=0.1cm] (3.25,0) -- (4,0);
\tikzmath{\h1=(2.75-(1+2*cos(60))+4.5+3;}
\vertex[blue]{1}[(\h1,0)]{({cos(60)},{sin(60)})}[']
\vertex[blue]{2}[(\h1,0)]{({1+cos(60)},{sin(60)})}[']
\vertex[red]{3}[(\h1,0)]{({1+2*cos(60)},0)}[']
\vertex[red]{4}[(\h1,0)]{({1+cos(60)},{-sin(60)})}[']
\vertex[blue]{5}[(\h1,0)]{({cos(60)},{-sin(60)})}[']
\vertex[blue]{6}[(\h1,0)]{(0,0)}[']
\vertex[blue]{7}[(\h1,0)]{(-1,0)}[']
\vertex[blue]{8}[(\h1,0)]{(-2,0)}[']
\vertex[blue]{9}[(\h1,0)]{(-3,0)}[']
\draw[blue] (v'5)--(v'6) (v'2)--(v'1)--(v'6)--(v'7)--(v'8)--(v'9);
\draw[red] (v'3)--(v'4);
\node[below right=0.2cm and 0.0cm of v'3] {$\textcolor{red}{\ttrip_1}$};
\node[above left=-0.3cm and 0.2cm of v'1] {$\textcolor{blue}{\ttrip_{\min}}$};
\node[text width=6cm, align=center] at ({(-3+2*cos(60)+1)/2+\h1-0.5},2) {${\forest_H\in H_1,\quad\forest_H=\rotatejoin_1(\forest_G)=}$\\$\left\langle\textcolor{red}{\Big(3,(2),2\Big)},\textcolor{blue}{\Big(5,(1,\textcolor{lightblue}{\textbf{2,4}}),1\Big)}\right\rangle$};
\end{tikzpicture}
\end{subfigure}\vspace{2mm}
\begin{subfigure}{\textwidth}
\centering
    \caption{Rotate-joining where $r'=|T'|+|T_{\min}|-k$.}
    \label{subfig:CaTk-rotatebj-3}
\begin{tikzpicture}
\vertex[blue]{1}{({cos(60)},{sin(60)})}
\vertex[blue]{2}{({1+cos(60)},{sin(60)})}
\vertex[lightblue]{3}{({1+2*cos(60)},0)}
\vertex[lightblue]{4}{({1+cos(60)},{-sin(60)})}
\vertex[lightblue]{5}{({cos(60)},{-sin(60)})}
\vertex[lightblue]{6}{(0,0)}
\vertex[lightblue]{7}{(-1,0)}
\vertex[lightblue]{8}{(-2,0)}
\vertex[lightblue]{9}{(-3,0)}
\draw[lightblue] (v3)--(v4)--(v5)--(v6)--(v7)--(v8)--(v9);
\draw[blue] (v1)--(v2);
\node[below left=0.2cm and 0.5cm of v6] {$\textcolor{lightblue}{\ttrip'}$};
\node[above left=-0.3cm and 0.2cm of v1] {$\textcolor{blue}{\ttrip_{\min}}$};
\node[align=center,text width=6.5cm] at ({(-3+2*cos(60)+1)/2},2) {${\forest_G\in G_{-1},\quad\forest_G=\rotatebreak_{-1}(\forest_H)=}$\\$\left\langle\textcolor{lightblue}{\Big(3,(\textbf{7}),5\Big)},\textcolor{blue}{\Big(1,(2),2\Big)}\right\rangle$};
\draw[<->, line width=0.1cm] (3.25,0) -- (4,0);
\tikzmath{\h1=(2.75-(1+2*cos(60))+4.5+3;}
\vertex[blue]{1}[(\h1,0)]{({cos(60)},{sin(60)})}[']
\vertex[blue]{2}[(\h1,0)]{({1+cos(60)},{sin(60)})}[']
\vertex[blue]{3}[(\h1,0)]{({1+2*cos(60)},0)}[']
\vertex[blue]{4}[(\h1,0)]{({1+cos(60)},{-sin(60)})}[']
\vertex[blue]{5}[(\h1,0)]{({cos(60)},{-sin(60)})}[']
\vertex[blue]{6}[(\h1,0)]{(0,0)}[']
\vertex[blue]{7}[(\h1,0)]{(-1,0)}[']
\vertex[blue]{8}[(\h1,0)]{(-2,0)}[']
\vertex[blue]{9}[(\h1,0)]{(-3,0)}[']
\draw[blue] (v'2)--(v'3)--(v'4)--(v'5)--(v'6)--(v'1) (v'9)--(v'8)--(v'7)--(v'6);
\node[above left=-0.3cm and 0.2cm of v'1] {$\textcolor{blue}{\ttrip_{\min}}$};
\node[text width=6.5cm, align=center] at ({(-3+2*cos(60)+1)/2+\h1-0.5},2) {${\forest_H\in H_{-1},\quad\forest_H=\rotatejoin_{-1}(\forest_G)=}$\\$\left\langle\textcolor{blue}{\Big(2,(2,\textcolor{lightblue}{\textbf{7}}),2\Big)}\right\rangle$};
\end{tikzpicture}
\end{subfigure}
\end{figure}
\paragraph{Step 5:}
The remaining forest triples are in sets $I_1$ and $I_2$, where
\begin{gather*}
    I_1=\{\forest:(a,1)\in E(T)\scolon\len(\alpha^{(1)})=\cdots=\len(\alpha^{(m)})=\len(\alpha^{(\min)})=1\scolon r_1,\ldots,r_{m-1}\geq 2\scolon\\
    \text{either $r_m\geq 2$ or $|T_m|\leq k-1$}\},\\
    I_2=\{\forest:(a,1)\not\in E(T)\scolon \len(\alpha^{(1)})=\cdots=\len(\alpha^{(m)})=\len(\alpha^{(\min)})=1\scolon r_1,\ldots,r_m\geq2\scolon\\
    r'\geq |T'|+|T_{\min}|-k+1\}.
\end{gather*}
First, note that if $\forest\in\FT'(C_a+U_k)$ does not have edge $(a,1)$, then it must have $\alpha_l'\geq k$. If $\len(\alpha')\geq 2$, then \begin{equation*}r'\leq \alpha_1'\leq |T'|-\alpha_l'\leq |T'|-k\leq |T'|+|T|-k,\end{equation*}
meaning $\forest\not \in I_2$. Thus, $I_1$ and $I_2$ contain only unit forest triples.

Combining the bijections on sets $A$ through $H$ as well as the identity function on $I_1$ and $I_2$ lets us construct a first-preserving involution on $\FT'(C_a+U_k)$, where the set of fixed points is $I_1\cup I_2$.

\paragraph{Computing Sum of Fixed Points:}
The last step is to find an expression for the sum of fixed points, by counting the forest triples in $I_1$ and $I_2$. We recall
\begin{equation*}
B_{a,k}^{(i)}(\boldsymbol{x})=\sum_{\forest\in\FT'^{(i)}(C_a+U_k)}\sign(\forest)\cdot e_{\type'(\forest)}=\sum_{\substack{\forest\in I_1\cup I_2\\\alpha_1^{(\min)}=i, r_{\min}=1}}\sign(\forest)\cdot e_{\type'(\forest)}.
\end{equation*}
For $\forest\in I_1$ (where edge $(a,1)\in E(T)$), we define $\comp'(\forest)=\alpha^{(m)}\cdots \alpha^{(1)}$. For $\forest\in I_2$ (where edge $(a,1)\not\in E(T)$), we define $\comp'(\forest)=\alpha'\cdot \alpha^{(m)}\cdots \alpha^{(1)}$. Note that $\type'(\forest)=\sort(\comp'(\forest))$, and $\comp'(\forest)\models a+k-\alpha_1^{(\min)}-1$.
\paragraph{Case 1:}
We will first find $B_{a,k}^{(i)}(\boldsymbol{x})$ where $i\leq k-1$. Let $\beta\models a+k-\alpha_1^{(\min)}-1$ be some composition, we want to count the number of forest triples where $\comp'(\forest)\models \beta$.

For $\forest\in I_1$, note that $\alpha_1^{(\min)}=\alpha_l^{(\min)}\geq k$ for $\forest\in \FT'(C_a+U_k)$, so $\alpha_1^{(\min)}\neq i$, and we have zero forest triples.

For $\forest\in I_2$, since edge $(a,1)\not \in E(T_{\min})$, there is only one way to place the tree triples so that $\comp'(\forest)=\beta$. Note that $\alpha'_1=\beta_1\geq k$. There are $|T_m|-1=\beta_2-1$ possible values for $r_m$, and $|T_{m-1}|-1=\beta_3-1$ possible values for $r_{m-1}$, and we use similar logic up until $r_1$. There are $k-|T_{\min}|=k-i$ possible values for $r'$. Thus, there are $(k-i)\cdot (\beta_2-1)\cdots (\beta_l-1)$ forest triples in $I_2$ where $\comp'(\forest)=\beta$. This means
\begin{equation*}
    B_{a,k}^{(i)}(\boldsymbol{x})=\displaystyle{\sum_{\substack{\beta\models a+k-i-1\\\beta_1\geq k}}}(k-i)\cdot (\beta_2-1)\cdots (\beta_l-1)\cdot e_{\sort(\beta)},\hspace{1cm}\text{\normalfont if }i\leq k-1.
\end{equation*}
\paragraph{Case 2:}
We will now find $B_{a,k}^{(i)}(\boldsymbol{x})$ where $i\geq k$. Let $\beta\models a+k-\alpha_1^{(\min)}-1$, we proceed like before.

First, for $\forest\in I_1$, there are $|T_{\min}|-k=i-k$ ways to place the first tree triple, and then one way to place the remaining tree triples. For $r_{m-1}$, there are $|T_{m-1}|-1=\beta_2-1$ possible values. For $r_{m-2}$, there are $\beta_3-1$ possible values, and this is true till $r_1$. If $|T_m|=\beta_1\leq k-1$, then there are $\beta_1$ possible values for $r_m$ (since $r_m$ can equal 1), otherwise there are $\beta_1-1$ possible values (since $r_m\neq 1$).

If $\forest\in I_2$, note that $r'\geq |T'|+|T_{\min}|-k+1$, but since this case $|T_{\min}|=i\geq k$, we get $r'\geq |T'|+1$. However, $\forest$ must have $r'\leq |T'|$, so there are no forest triples to count in this case. Thus, summing up the forest triples in $I_1$ gets
\begin{align*}
\begin{split}
    B_{a,k}^{(i)}(\boldsymbol{x})=&\sum_{\beta\models a+k-i-1}(i-k)\cdot (\beta_1-1)\cdots (\beta_l-1)\cdot e_{\sort(\beta)}\\+&\sum_{\substack{\beta\models a+k-i-1\\\beta_1\leq k-1}}(i-k)\cdot (\beta_2-1)\cdots (\beta_l-1)\cdot e_{\sort(\beta)}
\end{split}\\
    =&\sum_{\substack{\beta\models a+k-i\\\beta_1\leq k}}(i-k)\cdot (\beta_2-1)\cdots (\beta_l-1)\cdot e_{\sort(\beta_1-1,\beta\setminus\beta_1)},\hspace{1cm}\text{\normalfont if }i\geq k.
\end{align*}
The final step involves combining the two sums into a single sum over compositions $\beta\models a+k-i$, where if $\beta_1=1$, the sum gives the value of the top summation, while if $2\leq \beta_1\leq k$, the sum gives the value of the second summation. Note that if $i=k$, the sum becomes 0.

Thus, we have derived Equation~\ref{eqn:specBak}, and using Equation~\ref{eqn:nicebftobf} and some composition manipulation, we can derive Equation~\ref{eqn:bak}.
\end{proof}

\section{Further Directions}
\paragraph{Non-Adjacent Cycle Chains:}
\begin{figure}[b]
\centering
\caption{An $e-positive$ non-adjacent cycle graph.}
\label{fig:posnonadjacentcycle}
\begin{tikzpicture}
\tikzmath{\o1=0.5/sin(36);}
\vertex{1}{({\o1*cos(144)-1},{\o1*sin(144)})};
\vertex{5}{(0:\o1)};
\vertex[black][above]{6}{(72:\o1)};
\vertex{2}{(144:\o1)};
\vertex[black][left]{3}{(216:\o1)};
\vertex{4}{(288:\o1)};
\vertex{7}{(\o1+1,0)};
\vertex{8}{(\o1+2,0)};
\draw (v1)--(v2) (v2)--(v3)--(v4)--(v5)--(v6)--(v2) (v8)--(v7)--(v5);
\end{tikzpicture}
\end{figure}
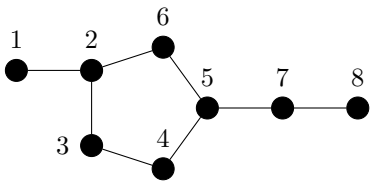
This paper proved that cycle chains connected at adjacent vertices are $e$-positive. However, this property does not hold in general for cycles connected at non-adjacent vertices. For example, in Figure~\ref{fig:fakecyclechains}, the cycle chain on the left (which contains cycles connected at adjacent vertices) is $e$-positive, while the graph on the right containing cycles connected at non-adjacent vertices is not (with negative coefficient $-8e_{3,2,2}$).

All non-adjacent cycle chains with 6 and 7 vertices are not $e$-positive, but there are some non-adjacent cycle chains that are $e$-positive, with the smallest example being the graph in Figure~\ref{fig:posnonadjacentcycle}. One possible direction is finding out the exact conditions for when cycle chains are $e$-positive.

\paragraph{Attaching at Cut Vertices:}
We can generalize the method used for adjacent cycle chains to both non-adjacent cycle chains and more generally graphs attached at cut vertices. Say $G'$ is an arbitrary graph and $A$ is a graph with $a$ vertices. Then, we define $B_{A,k}(\boldsymbol{x})$ similarly to how we defined it in Section~\ref{sec:cycchain}, but instead looking at subsets of $\FT(A+U_k)$. The chromatic symmetric function of $A+G'$ can be written similar to Equation~\ref{eqn:csfC_a+G'} as
\begin{equation*}
    X_{A+G'}(\boldsymbol{x})=\sum_{k=1}^{|G'|}X^{(k)}_{G'}(\boldsymbol{x})\cdot B_{A,k}(\boldsymbol{x}),
\end{equation*}
using the results from Section~\ref{sec:cycchain}. Note that $G'$ may not have a first-preserving involution, in which case the value of $X^{(k)}_{G'}$ would be negative for some integer $k$.

If $G'$ does have a first-preserving involution and $B_{A,k}(\boldsymbol{x})$ is non-negative for all integers $k$, then the graph $A+G'$ is $e$-positive. Moreover, if $B_{A,k}^{(i)}$ is always $e$-positive for integer pairs $k,i$, then $A+G'$ also has a first-preserving involution. This paper proved this true for $A=C_a$, shown in Section~\ref{sec:cpti}, and a similar method could be used to prove that chains formed by connecting other graphs at single vertices are $e$-positive.
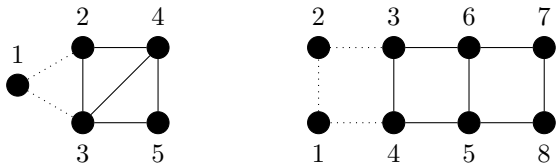
\begin{figure}[!htb]
\centering
\caption{Attaching a vertex or a cycle to a graph at two vertices.}
\label{fig:attachat2}
\begin{tikzpicture}
\tikzmath{\o1=-4;}
\vertex{1}[(\o1,0)]{(0,0)}['];
\vertex{2}[(\o1,0)]{({sqrt(3)/2},0.5)}[']
\vertex[black][below]{3}[(\o1,0)]{({sqrt(3)/2},-0.5)}[']
\vertex[black][below]{5}[(\o1,0)]{({1+sqrt(3)/2},-0.5)}[']
\vertex{4}[(\o1,0)]{({1+sqrt(3)/2},0.5)}[']
\draw[dotted] (v'2)--(v'1)--(v'3);
\draw (v'2)--(v'3)--(v'5)--(v'4)--(v'2) (v'3)--(v'4);
\vertex[black][below]{1}{(0,-0.5)};
\vertex{2}{(0,0.5)};
\vertex{3}{(1,0.5)};
\vertex[black][below]{4}{(1,-0.5)};
\vertex[black][below]{5}{(2,-0.5)};
\vertex{6}{(2,0.5)};
\vertex{7}{(3,0.5)};
\vertex[black][below]{8}{(3,-0.5)};
\draw[dotted] (v4)--(v1)--(v2)--(v3);
\draw (v3)--(v4)--(v5)--(v6)--(v7)--(v8)--(v5) (v3)--(v6);
\end{tikzpicture}
\end{figure}
\paragraph{Attaching at Multiple Vertices:}
Another further direction could be attaching a vertex or a cycle at multiple edges. The graph on the left of Figure~\ref{fig:attachat2} shows how to attach a vertex to the triangle ladder with 4 vertices, using dotted edges. The right graph of Figure~\ref{fig:attachat2} shows attaching $C_4$ to a chain of cycles connected at two edges. Both graphs are $e$-positive; finding when it is possible to attach a vertex or cycle to two edges of a graph with a first-preserving involution is another further direction to explore. 


\section{Acknowledgements}

We would like to thank the USA--PRIMES program for providing the opportunity to work on this topic.
\bibliographystyle{hplain} 
\bibliography{main} 
\end{document}